\documentclass{birkjour}

\newtheorem{thm}{Theorem}[section]
\newtheorem{cor}[thm]{Corollary}
\newtheorem{lem}[thm]{Lemma}
\newtheorem{prop}[thm]{Proposition}
\theoremstyle{definition}

\theoremstyle{remark}

\numberwithin{equation}{section}

\usepackage{amsmath,amsfonts,amssymb,amsxtra,setspace,xspace,graphicx,lmodern,fourier,mathrsfs}
\usepackage[colorlinks=true]{hyperref}
\hypersetup{urlcolor=blue, citecolor=red, linkcolor=blue}
\usepackage[usenames,dvipsnames]{color}

\newcommand{\R}{{\mathbb R}}

\newcommand{\N}{{\mathbb N}}
\newcommand{\be}[1]{\begin{equation}\label{#1}}
\newcommand{\ee}{\end{equation}}
\renewcommand{\(}{\left(}
\renewcommand{\)}{\right)}

\newcommand{\ird}[1]{\int_{\R^d}{#1}\,dx}

\newcommand{\nrm}[2]{\|{#1}\|_{#2}}
\newcommand{\irtwo}[1]{\int_{\R^2}{#1}\,dx}

\newcommand{\istwo}[1]{\int_0^{+\infty}{#1}\,r\,dr}

\newcommand{\pmin}{1}
\newcommand{\qmin}{1}
\newcommand{\LLL}{$\mathrm{LLL}$}

\begin{document}

\title[Interpolation inequalities and spectral estimates for magnetic operators]{Interpolation inequalities and spectral\\ estimates for magnetic operators}

\author[J.~Dolbeault]{Jean Dolbeault}
\address{CEREMADE (CNRS UMR n$^\circ$ 7534), PSL research university, Universit\'e Paris-Dauphine, Place de Lattre de Tassigny, 75775 Paris 16, France}
\email{dolbeaul@ceremade.dauphine.fr}

\author[M.J.~Esteban]{Maria J.~Esteban}
\address{CEREMADE (CNRS UMR n$^\circ$ 7534), PSL research university, Universit\'e Paris-Dauphine, Place de Lattre de Tassigny, 75775 Paris 16, France}
\email{esteban@ceremade.dauphine.fr}

\author[A.~Laptev]{Ari Laptev}
\address{Department of Mathematics, Imperial College London, Huxley Building, 180 Queen's Gate, London SW7 2AZ, UK,\\
and Department of Mathematics, Siberian Federal University, Russia}
\email{a.laptev@imperial.ac.uk}

\author[M.~Loss]{Michael Loss}
\address{School of Mathematics, Skiles Building, Georgia Institute of Technology, Atlanta GA 30332-0160, USA}
\email{loss@math.gatech.edu}

\thanks{\hspace*{0pt}Acknowledgements. This work has been done at the Institute Mittag-Leffler during the fall program \emph{Interactions between Partial Differential Equations \& Functional Inequalities}. J.D.~has been supported by the Projects \emph{STAB}, \emph{Kibord} and \emph{EFI} of the French National Research Agency (ANR). AL was supported by the grant of the Russian Federation Government for research, contract No. 14.Y26.31.0006. M.L. has been partially supported by NSF Grant DMS-1600560. The authors thank two anonymous referees who helped them clarifying the assumptions needed for the existence of optimal functions and for pointing references concerning the stability among radial solutions.\\
\copyright\,2017 by the authors. This paper may be reproduced, in its entirety, for non-commercial purposes.}

%%%%%%%%%%%%%%%%%%%%%%%%%%%%%%%%%%%%%%%%%%%%%%%%%%%%%%%%%%%%%%%%%%%%%%
\date{\today}
\begin{abstract}
We prove magnetic interpolation inequalities and Keller-Lieb-Thir\-ring estimates for the principal eigenvalue of magnetic Schr\"odinger operators. We establish explicit upper and lower bounds for the best constants and show by numerical methods that our theoretical estimates are accurate.
\end{abstract}
\maketitle
\thispagestyle{empty}
\vspace*{-0.5cm}
%%%%%%%%%%%%%%%%%%%%%%%%%%%%%%%%%%%%%%%%%%%%%%%%%%%%%%%%%%%%%%%%%%%%%%

%%%%%%%%%%%%%%%%%%%%%%%%%%%%%%%%%%%%%%%%%%%%%%%%%%%%%%%%%%%%%%%%%%%%%%
%%%%%%%%%%%%%%%%%%%%%%%%%%%%%%%%%%%%%%%%%%%%%%%%%%%%%%%%%%%%%%%%%%%%%%
\section{Introduction and main results}\label{introduction}

In dimensions $d=2$ and $d=3$, let us consider the magnetic Laplacian defined via a magnetic potential $\mathbf A$ by
\[
-\Delta_{\mathbf A}\,\psi=-\Delta\,\psi-\,2\,i\,\mathbf A\cdot\nabla\psi+|\mathbf A|^2\psi-\,i\,(\mbox{div}\,\mathbf A)\,\psi\,.
\]
The magnetic field is $\mathbf B=\mbox{curl}\,\mathbf A$. The quadratic form associated with $-\Delta_{\mathbf A}$ is given by
$\ird{|\nabla_{\!\mathbf A}\psi|^2}$ and well defined for all functions in the space
\[
\mathrm H^1_{\!\mathbf A}(\R^d):=\left\{\psi\in\mathrm L^2(\R^d)\,:\,\nabla_{\!\mathbf A}\psi\in\mathrm L^2(\R^d)\right\}
\]
where
\[
\nabla_{\!\mathbf A}:=\nabla+\,i\,\mathbf A\,.
\]
We shall consider the following spectral gap inequality
\be{Gap}
\nrm{\nabla_{\!\mathbf A}\psi}2^2\ge\Lambda[\mathbf B]\,\nrm\psi2^2\quad\forall\,\psi\in\mathrm H^1_{\!\mathbf A}(\R^d)\,.
\ee
Let us notice that $\Lambda$ depends only on $\mathbf B=\mbox{curl}\,\mathbf A$. Throughout this paper, we shall assume that there is equality in~\eqref{Gap} for some function in $\mathrm H^1_{\!\mathbf A}(\R^d)$. If $\mathbf B$ is a constant magnetic field, we recall that $\Lambda[\mathbf B]=|\mathbf B|$. If $d=2$, the spectrum of $-\Delta_{\mathbf A}$ is the countable set $\{(2j+1)\,|\mathbf B|\,:\,j\in\N\}$, the eigenspaces are of infinity dimension and called the \emph{Landau levels}. The eigenspace corresponding to the lowest level ($j=0$) is called the \emph{Lowest Landau Level} and will be considered in Section~\ref{Sec:Asymptotic}.

Let us denote the critical Sobolev exponent by $2^*=+\infty$ if $d=2$ and $2^*\kern-3pt=6$ if $d=3$, and define the optimal Gagliardo-Nirenberg constant by
\be{defcp}
\mathsf C_p:=\left\{\begin{array}{cc}
\min_{u\in\mathrm H^1(\R^d)\setminus\{0\}}\frac{\nrm{\nabla u}2^2+\nrm u2^2}{\nrm up^2}\quad&\mbox{if}\quad p\in(2,2^*)\,,\\[6pt]
\min_{u\in\mathrm H^1(\R^d)\setminus\{0\}}\frac{\nrm{\nabla u}2^2+\nrm up^2}{\nrm u2^2}\quad&\mbox{if}\quad p\in(1,2)\,.
\end{array}\right.
\ee
The first purpose of this paper is to establish interpolation inequalities in the presence of a magnetic field. With $\mathbf A$ and $\mathbf B=\mbox{curl}\,\mathbf A$ as above, such that~\eqref{Gap} holds, let us consider the \emph{magnetic interpolation inequalities}
\be{Interp1}
\nrm{\nabla_{\!\mathbf A}\psi}2^2+\alpha\,\nrm\psi2^2\ge\mu_{\mathbf B}(\alpha)\,\nrm\psi p^2\quad\forall\,\psi\in\mathrm H^1_{\!\mathbf A}(\R^d)
\ee
for any $\alpha\in(-\Lambda[\mathbf B],+\infty)$ and any $p\in(2,2^*)$,
\be{Interp2}
\nrm{\nabla_{\!\mathbf A}\psi}2^2+\beta\,\nrm\psi p^2\ge\nu_{\mathbf B}(\beta)\,\nrm\psi2^2\quad\forall\,\psi\in\mathrm H^1_{\!\mathbf A}(\R^d)
\ee
for any $\beta\in(0,+\infty)$ and any $p\in(1,2)$ and, in the limit case corresponding to $p=2$,
\be{Interp3}
\nrm{\nabla_{\!\mathbf A}\psi}2^2\ge\gamma\ird{|\psi|^2\,\log\(\frac{|\psi|^2}{\nrm\psi2^2}\)}+\xi_{\mathbf B}(\gamma)\,\nrm\psi2^2\quad\forall\,\psi\in\mathrm H^1_{\!\mathbf A}(\R^d)
\ee
for any $\gamma\in(0,+\infty)$. Throughout this paper $\mu_{\mathbf B}(\alpha)$, $\nu_{\mathbf B}(\beta)$ and $\xi_{\mathbf B}(\gamma)$ denote the \emph{optimal constants} in, respectively,~\eqref{Interp1},~\eqref{Interp2} and~\eqref{Interp3}, considered as functions of the parameters $\alpha$, $\beta$ and $\gamma$. We observe that $\mu_{\mathbf 0}(1)=\mathsf C_p$ if $p\in(2,2^*)$, $\nu_{\mathbf 0}(1)=\mathsf C_p$ if $p\in(1,2)$ and $\xi_{\mathbf 0}(\gamma)=\gamma\,\log\big(\pi\,e^2/\gamma\big)$ if $p=2$ (which is the classical constant in the Euclidean logarithmic Sobolev inequality: see~\eqref{LS}). We shall assume that the magnetic potential $\mathbf A\in\mathrm L^2_{\mathrm{loc}}(\R^d)$ satisfies the technical assumption
\be{hyppp}
\begin{array}{l}
\displaystyle\lim_{\sigma\to+\infty}\sigma^{d-2}\ird{|\mathbf A(x)|^2\,e^{-\sigma\,|x|}}=0\quad\mbox{if}\quad p\in(2,2^*)\,,\\[8pt]
\displaystyle\lim_{\sigma\to+\infty}\frac{\sigma^{\frac d2-1}}{\log\sigma}\ird{|\mathbf A(x)|^2\,e^{-\sigma\,|x|^2}}=0\quad\mbox{if}\quad p=2\,,\\[8pt]
\displaystyle\lim_{\sigma\to+\infty}\sigma^{d-2}\int_{|x|<1/\sigma}|\mathbf A(x)|^2\,dx\quad\mbox{if}\quad p\in(1,2)\,.
\end{array}
\ee
%---------------------------------------------------------------------
\begin{thm}\label{Thm:Interp} Assume that $d=2$ or $3$, $p\in(1,2)\cup(2,2^*)$, and $\alpha>2$ if $d=2$ or $\alpha=3$ if $d=3$. Let $\mathbf A\in\mathrm L^\alpha_{\mathrm{loc}}(\R^d)$ be a magnetic potential satisfying~\eqref{hyppp} and $\mathbf B=\mbox{curl}\,\mathbf A$ be a magnetic field on $\R^d$ such that~\eqref{Gap} holds for some $\Lambda=\Lambda[\mathbf B]>0$ and equality is achieved in~\eqref{Gap} for some function $\psi\in\mathrm H^1_{\mathbf A}(\R^d)$. Then, the following properties hold:
\begin{enumerate}
\item[(i)] For any $p\in(2,2^*)$, the function $\mu_{\mathbf B}:(-\Lambda,+\infty)\to(0,+\infty)$ is monotone increasing, concave and such that
\[
\lim_{\alpha\to(-\Lambda)_+}\mu_{\mathbf B}(\alpha)=0\quad\mbox{and}\quad\lim_{\alpha\to+\infty}\mu_{\mathbf B}(\alpha)\,\alpha^{\frac{d-2}2-\frac dp}=\mathsf C_p\,.
\]
\item[(ii)] For any $p\in(1,2)$, the function $\nu_{\mathbf B}:(0,+\infty)\to(\Lambda,+\infty)$ is monotone increasing, concave and such that
\[
\lim_{\beta\to0_+}\nu_{\mathbf B}(\beta)=\Lambda\quad\mbox{and}\quad\lim_{\beta\to+\infty}\nu_{\mathbf B}(\beta)\,\beta^{-\frac{2\,p}{2\,p+d\,(2-p)}}=\mathsf C_p\,.
\]
\item[(iii)] The function $\xi_{\mathbf B}:[0,+\infty)\to\R$ is continuous, concave, such that $\xi_{\mathbf B}(0)=\Lambda[\mathbf B]$ and
\[
\xi_{\mathbf B}(\gamma)=\tfrac d2\,\gamma\,\log\big(\tfrac{\pi\,e^2}\gamma\big)(1+o(1))\quad\mbox{as}\quad\gamma\to+\infty\,.
\]
\end{enumerate}
\end{thm}
%---------------------------------------------------------------------
Equality is achieved in~\eqref{Interp1},~\eqref{Interp2} and~\eqref{Interp3} for some $\psi\in\mathrm H^1_{\!\mathbf A}(\R^d)$ in the case of constant magnetic fields. In the case of nonconstant magnetic fields, there are cases where one can prove the existence of some $\psi\in\mathrm H^1_{\!\mathbf A}(\R^d)$ for which equality is achieved in~\eqref{Interp1},~\eqref{Interp2} and~\eqref{Interp3}, but general sufficient conditions are difficult to obtain. Some answers to this question can be found in~\cite[Section~4]{MR1034014} and in~\cite{Kurata-2000}.

The main result of this paper is to establish lower bounds for the \emph{optimal constants} $\mu_{\mathbf B}$, $\nu_{\mathbf B}$ and~$\xi_{\mathbf B}$ in the case of general magnetic fields (respectively in Propositions~\ref{T1-gen},~\ref{T1-gen-qsmall} and in Section~\ref{ProofOfTheorem1}) and in the case of two-dimensional constant magnetic fields (respectively in Propositions~\ref{T1-gen-bis},~\ref{T1-gen-pless23} and~\ref{Prop:LS}). Upper estimates, theoretical and numerical, are also given in Section~\ref{Sec:Numerics}.

\medskip The magnetic interpolation inequalities have interesting applications to optimal spectral estimates for the magnetic Schr\"odinger operators
\[
-\Delta_{\mathbf A}+\phi\,.
\]
Let us denote by $\lambda_{\mathbf A,\phi}$ its principal eigenvalue, and by $\alpha_{\mathbf B}:(0,+\infty)\to(-\Lambda,+\infty)$ the inverse function of $\alpha\mapsto\mu_{\mathbf B}(\alpha)$. We denote by $\phi_-:=(\phi-|\phi|)/2$ the negative part of~$\phi$. By duality as we shall see in Section~\ref{Sec:proofs}, Theorem~\ref{Thm:Interp} has a counterpart, which is a result on \emph{magnetic Keller-Lieb-Thirring estimates}.
%---------------------------------------------------------------------
\begin{cor}\label{Cor:KLT1} With these notations, let us assume that $\mathbf A$ satisfies the same hypotheses as in Theorem~\ref{Thm:Interp}. Then we have:
\begin{enumerate}
\item[(i)] For any $q=p/(p-2)\in(d/2,+\infty)$ and any potential $V$ such that $V_-\in\mathrm L^q(\R^d)$,
\be{KLT1}
\lambda_{\mathbf A,V}\ge-\,\alpha_{\mathbf B}(\nrm{V_-}q)\,.
\ee
The function $\alpha_{\mathbf B}$ satisfies
\[
\lim_{\mu\to0_+}\alpha_{\mathbf B}(\mu)=\Lambda\quad\mbox{and}\quad\lim_{\mu\to+\infty}\alpha_{\mathbf B}(\mu)\,\mu^\frac{2\,(q+1)}{d-2-2\,q}=-\,\mathsf C_p^\frac{2\,(q+1)}{d-2-2\,q}\,.
\]
\item[(ii)] For any $q=p/(2-p)\in(1,+\infty)$ and any potential $W\ge 0$ such that $W^{-1}\in\mathrm L^q(\R^d)$,
\be{KLT2}
\lambda_{\mathbf A,W}\ge\nu_{\mathbf B}\(\nrm{W^{-1}}q^{-1}\)\,.
\ee
\item[(iii)] For any $\gamma>0$ and any potential $W\ge 0$ such that $e^{-W/\gamma}\in\mathrm L^1(\R^d)$,
\be{KLT3}
\lambda_{\mathbf A,W}\ge\xi_{\mathbf B}\(\gamma\)-\gamma\,\log\(\textstyle\ird{e^{-W/\gamma}}\)\,.
\ee
\end{enumerate}
Moreover equality is achieved in~\eqref{KLT1},~\eqref{KLT2} and~\eqref{KLT3} if and only if equality is achieved in~\eqref{Interp1},~\eqref{Interp2} and~\eqref{Interp3}.\end{cor}
%---------------------------------------------------------------------
For general potentials changing sign, a more general estimate is proved in Proposition~\ref{Prop:keller-Case3}. A first result without magnetic field was obtained by Keller in the one-dimensional case in~\cite{MR0121101}, before being rediscovered and extended to the sum of all negative eigenvalues in any dimension by Lieb and Thirring in~\cite{Lieb-Thirring76}. In the meantime, an estimate similar to~\eqref{KLT3} was established in~\cite{Federbush} which, by duality, provides a proof of the logarithmic Sobolev inequality given by Gross in~\cite{Gross75}. In the Euclidean framework without magnetic fields, scalings provide a scale invariant form of the inequality, which is stronger (see~\cite{MR479373,MR3493423}) but was already known as the Blachmann-Stam inequality and goes back at least to~\cite{MR0109101}: see~\cite{Villani2008,MR3255069} for an historical account. Many papers have been devoted to the issue of estimating the optimal constants for the so-called Lieb-Thirring inequalities: see for instance~\cite{MR1708811,MR2253013,MR2443931} for estimates on the Euclidean space,~\cite{MR3218815,Dolbeault2013437} in the case of compact manifolds, and~\cite{MR3570296} for non-compact manifolds (infinite cylinders). As far as we know, no systematic study as in Theorem~\ref{Thm:Interp} nor as in Corollary~\ref{Cor:KLT1} has been done so far in the presence of a magnetic field, although many partial results have been previously obtained using, \emph{e.g.}, the diamagnetic inequality.

Section~\ref{Sec:proofs} is devoted to the duality between Theorem~\ref{Thm:Interp} and Corollary~\ref{Cor:KLT1}. Most of our paper is devoted to estimates of the best constants in~\eqref{Interp1},~\eqref{Interp2} and~\eqref{Interp3}, which also provide estimates of the best constants in~\eqref{KLT1},~\eqref{KLT2} and~\eqref{KLT3}. In Section~\ref{general-magnetic-ineq} we prove lower estimates in the case of a general magnetic field and establish Theorem~\ref{Thm:Interp}. Sharper estimates are obtained in Section~\ref{constant-magnetic-ineq} for a constant magnetic field in dimension two. Section~\ref{Sec:Numerics} is devoted to upper bounds and the numerical computation of various upper and lower bounds (constant magnetic field, dimension two). Our theoretical estimates are remarkably accurate for the values of $p$ and $d$ that we have considered numerically, using radial functions. This is why we conclude this paper by a numerical investigation of the stability of a radial optimal function.

%%%%%%%%%%%%%%%%%%%%%%%%%%%%%%%%%%%%%%%%%%%%%%%%%%%%%%%%%%%%%%%%%%%%%%
%%%%%%%%%%%%%%%%%%%%%%%%%%%%%%%%%%%%%%%%%%%%%%%%%%%%%%%%%%%%%%%%%%%%%%
\section{Magnetic interpolation inequalities and Keller-Lieb-Thirring inequalities: duality and a generalization}\label{Sec:proofs}

Let us prove Corollary~\ref{Cor:KLT1} as a consequence of Theorem~\ref{Thm:Interp}. Details on \emph{duality} will be provided in the proof and in the subsequent comments.

\begin{proof}[Proof of Corollary~\ref{Cor:KLT1}] Consider first Case (i) with $q>d/2$. Using the definition of the negative part of $V$ and H\"older's inequality with $1/q+2/p=1$, we know that
\begin{eqnarray}\label{SchrHolder}
\ird{|\nabla_{\!\mathbf A}\psi|^2}+\ird{V\,|\psi|^2}&\ge&\ird{|\nabla_{\!\mathbf A}\psi|^2}+\ird{V_-\,|\psi|^2}\\
&\ge&\nrm{\nabla_{\!\mathbf A}\psi}2^2-\nrm {V_-}q\,\nrm\psi p^2\ge-\,\alpha_{\mathbf B}(\nrm {V_-}q)\,{\nrm\psi 2^2}\,,\nonumber
\end{eqnarray}
because, by Theorem~\ref{Thm:Interp}, $\mu_{\mathbf B}(\alpha)=\nrm {V_-}q$ has a unique solution $\alpha=\alpha_{\mathbf B}(\nrm {V_-}q)$. This proves~\eqref{KLT1}. The optimality in~\eqref{KLT1} is equivalent to the optimality in~\eqref{Interp1} because $V=-\,|\psi|^{p-2}$ realizes the equality in H\"older's inequality.

In Case (ii), by H\"older's inequality with exponents $2/(2-p)$ and $2/p$,
\[
\nrm\psi p^2=\(\ird{W^{-\frac p2}\,\(W\,|\psi|^2\)^\frac p2}\)^{2/p}\le\nrm{W^{-1}}q\,\ird{W\,|\psi|^2}
\]
with $q=p/(2-p)$, we know using~\eqref{Interp2} that
\[
\ird{|\nabla_{\!\mathbf A}\psi|^2}+\ird{W\,|\psi|^2}\ge\ird{|\nabla_{\!\mathbf A}\psi|^2}+\beta\,\nrm\psi p^2\ge\nu_{\mathbf B}(\beta)\,\ird{|\psi|^2}\,.
\]
with $\beta=1/\nrm{{W}^{-1}}q$, which proves~\eqref{KLT2}.

In Case (iii), let us consider
\[
\mathcal F[\psi,W]:=\ird{|\nabla_{\!\mathbf A}\psi|^2}+\ird{W\,|\psi|^2}+\gamma\,\log\(\ird{e^{-W/\gamma}}\)-\,\xi_{\mathbf B}(\gamma)
\]
for a given function $\psi\in\mathrm H^1_{\!\mathbf A}(\R^d)$ such that $\nrm\psi2=1$ and mimimize this functional with respect to the potential $W$, so that
\[
|\psi|^2=\frac{e^{-W/\gamma}}{\ird{e^{-W/\gamma}}}
\]
which implies $W=W_\psi:=-\,\gamma\,\log|\psi|^2-\,\gamma\,\log\(\ird{e^{-W/\gamma}}\)$. Hence
\[
\mathcal F[\psi,W]\ge\mathcal F[\psi,W_\psi]=\ird{|\nabla_{\!\mathbf A}\psi|^2}-\gamma\ird{|\psi|^2\,\log\(|\psi|^2\)} -\,\xi_{\mathbf B}(\gamma)\ge0\,,
\]
where the last inequality is given by~\eqref{Interp3}. Minimizing $\mathcal F[\psi,W]$ with respect to~$W$ under the condition $\nrm\psi2=1$ establishes~\eqref{KLT3}. It is straightforward that the equality case is given by the equality case in~\eqref{Interp3} when there is a function $\psi$ for which this equality holds.
\end{proof}

In Case (iii) of Theorem~\ref{Thm:Interp} and Corollary~\ref{Cor:KLT1}, the \emph{duality} relation of~\eqref{Interp3} and~\eqref{KLT3} is a straightforward consequence of the convexity inequality
\[
x\,y+y\,\log y-y+e^{-x}\ge0\quad\forall\,(x,y)\in\R\times(0,+\infty)\,.
\]
A similar observation can be done in Cases (i) or (ii). If $q=p/(p-2)\in(d/2,+\infty)$, \emph{i.e.}, in Case (i), for an arbitrary negative potential~$V$ and an arbitrary function $\psi\in\mathrm H^1_{\!\mathbf A}(\R^d)$, we can rewrite~\eqref{SchrHolder} as
\[
\ird{|\nabla_{\!\mathbf A}\psi|^2}+\ird{V\,|\psi|^2}+\,\alpha_{\mathbf B}(\nrm Vq)\,{\nrm\psi 2^2}\ge0\,.
\]
By minimizing with respect to either $V$ or $\psi$, we reduce the inequality to~\eqref{Interp1} or~\eqref{KLT1}, and in both cases $V=-\,|\psi|^{p-2}$ is optimal. The two estimates are henceforth \emph{dual} of each other, which is reflected by the fact that $p/2$ and $q$ are H\"older conjugate exponents. Similarly in Case (ii), if $q=p/(2-p)\in(1,+\infty)$, we have
\[
\ird{|\nabla_{\!\mathbf A}\psi|^2}+\ird{W\,|\psi|^2}-\nu_{\mathbf B}(\beta)\,\ird{|\psi|^2}\ge0
\]
for any positive potential $W$ and any $\psi\in\mathrm H^1_{\!\mathbf A}(\R^d)$. Again a minimization with respect to either $W$ or $\psi$ reduces the inequality to~\eqref{Interp2} or~\eqref{KLT2}, which are also dual of each other. With these observations, it is clear that Theorem~\ref{Thm:Interp} can be proved as a consequence of Corollary~\ref{Cor:KLT1}: the two results are actually equivalent.

\medskip The restriction to a negative potential $V$ or to its negative part (resp.~to a positive potential $W$) is artificial in the sense that we can put the threshold at an arbitrary level $\lambda$. Let us consider a general potential $\phi$ on $\R^d$. We can first rewrite~\eqref{SchrHolder} in a more general setting as
\begin{multline*}
\ird{|\nabla_{\!\mathbf A}\psi|^2}+\ird{\phi\,|\psi|^2}\\
\ge\ird{|\nabla_{\!\mathbf A}\psi|^2}-\ird{(\lambda-\phi)_+\,|\psi|^2}+\lambda\ird{|\psi|^2}
\end{multline*}
with $\lambda\in\R$, $\mu=\nrm{(\lambda-\phi)}{q,+}$ and $q=p/(p-2)$. Here $\nrm u{q,+}$ is a new notation which stands for
\[
\nrm u{q,+}:=\(\int_{u>0}u^q\,dx\)^{1/q}\,.
\]
Using~\eqref{KLT1}, we know that
\[
\ird{|\nabla_{\!\mathbf A}\psi|^2}+\ird{\phi\,|\psi|^2}\ge-\(\alpha_{\mathbf B}(\mu)-\lambda\)\ird{|\psi|^2}\,.
\]
This makes sense of course if $\mu$ is finite and well defined which, for instance, requires that
\[
\lambda\le\lim_{R\to+\infty}\mathop{\mathrm{infess}}\displaylimits_{|x|>R}\,\phi(x)\,.
\]
A similar estimate holds in the range $p\in(\pmin,2)$. Let $\lambda\le\mathop{\mathrm{infess}}\displaylimits_{x\in\R^d}\,\phi(x)$. Then we have
\[
\nrm\psi p^2=\(\ird{(\phi-\lambda)^{-\frac p2}\,\((\phi-\lambda)\,|\psi|^2\)^\frac p2}\)^{2/p}\le\tfrac1\beta\,\ird{(\phi-\lambda)\,|\psi|^2}\,,
\]
with $1/\beta=\nrm{(\phi-\lambda)^{-1}}q$ and $q=p/(2-p)$. Using~\eqref{KLT2}, we know that
\[
\ird{|\nabla_{\!\mathbf A}\psi|^2}+\ird{\phi\,|\psi|^2}\ge\ird{|\nabla_{\!\mathbf A}\psi|^2}+\beta\,\nrm\psi p^2+\lambda\,\nrm\psi 2^2\ge\(\nu_{\mathbf B}(\beta)+\lambda\)\,\nrm\psi 2^2\,.
\]
We can collect these estimates in the following result.
%---------------------------------------------------------------------
\begin{prop}\label{Prop:keller-Case3} Let $d=2$ or $3$. Let $\phi\in\mathrm L^1_{\mathrm{loc}}(\R^d)$ be an arbitrary potential.
\begin{enumerate}
\item[(i)] If $q>d/2$, $p=\frac{2\,q}{q-1}$ and $\alpha_{\mathbf B}$ is defined as in~\eqref{KLT1}, we have
\[
\lambda_{\mathbf A,\phi}\ge-\(\alpha_{\mathbf B}\(\nrm{(\lambda-\phi)}{q,+}\)-\lambda\)\,.
\]
\item[(ii)] If $q\in(\qmin,+\infty)$, $p=\frac{2\,q}{q+1}$ and $\nu_{\mathbf B}$ defined as in~\eqref{KLT2}, we have
\[
\lambda_{\mathbf A,\phi}\ge\lambda+\nu_{\mathbf B}\(\nrm{(\phi-\lambda)^{-1}}q^{-1}\)\,.
\]
\end{enumerate}
These estimates hold for any $\lambda\in\R$ such that all above norms are well defined, with the additional condition that $\phi\ge\lambda$ a.e.~in Case (ii).\end{prop}
%---------------------------------------------------------------------
Notice that weaker conditions than $\phi\ge\lambda$ a.e.~can be given, like, for instance, $\inf_{\psi\in\mathrm H^1_{\!\mathbf A}(\R^d)}\int_{(\phi-\lambda)<0}\(|\nabla_{\!\mathbf A}\psi|^2+(\phi-\lambda)\,|\psi|^2\)\,dx\ge0$. Details are left to the reader. In Corollary~\ref{Cor:KLT1}, Case~(iii) does not involve a threshold at level $\lambda=0$ and one can notice that the estimate~\eqref {KLT3} is invariant under the transformation $\phi\mapsto\phi-\lambda$, $\lambda_{\mathbf A,\phi}\mapsto\lambda_{\mathbf A,\phi-\lambda}=\lambda_{\mathbf A,\phi}-\lambda$.

%%%%%%%%%%%%%%%%%%%%%%%%%%%%%%%%%%%%%%%%%%%%%%%%%%%%%%%%%%%%%%%%%%%%%%
%%%%%%%%%%%%%%%%%%%%%%%%%%%%%%%%%%%%%%%%%%%%%%%%%%%%%%%%%%%%%%%%%%%%%%
\section{Lower estimates: general magnetic field}\label{general-magnetic-ineq}

In this section, we consider a general magnetic field in dimension $d=2$ or $3$. We establish lower estimates of the best constants in~\eqref{Interp1},~\eqref{Interp2} and~\eqref{Interp3} before proving Theorem~\ref{Thm:Interp}.

%%%%%%%%%%%%%%%%%%%%%%%%%%%%%%%%%%%%%%%%%%%%%%%%%%%%%%%%%%%%%%%%%%%%%%
\subsection{Preliminaries: interpolation inequalities without magnetic field}\label{GN-without-magnetic}

Assume that $p>2$ and let $\mathsf C_p$ denote the optimal constant defined in~\eqref{defcp}, that is, the best constant in the Gagliardo-Nirenberg inequality
\be{UnScaledGN}
\nrm{\nabla u}2^2+\nrm u2^2\ge\mathsf C_p\,\nrm up^2\quad\forall\,u\in\mathrm H^1(\R^d)\,.
\ee
By scaling, if we test~\eqref{UnScaledGN} by $u\big(\cdot/\lambda\big)$, we find that
\be{ScaledGN}
\nrm{\nabla u}2^2+\lambda^2\,\nrm u2^2\ge\mathsf C_p\,\lambda^{2-\,d\,(1-\frac2p)}\,\nrm up^2\quad\forall\,u\in\mathrm H^1(\R^d)\quad\forall\,\lambda>0\,.
\ee
An optimization on $\lambda>0$ shows that the best constant in the scale-invariant inequality
\be{GN}
\nrm{\nabla u}2^{d\,(1-\frac2p)}\,\nrm u2^{2-d\,(1-\frac2p)}\ge\mathsf S_p\,\nrm up^2\,\quad\forall\,u\in\mathrm H^1(\R^d)
\ee
is given by
\be{defsp1}
\mathsf S_p=\tfrac1{2\,p}\,\(2\,p-d\,(p-2)\)^{1-d\,\frac{p-2}{2\,p}}\,\(d\,(p-2)\)^\frac{d\,(p-2)}{2\,p}\,\mathsf C_p\,.
\ee

Next, let us consider the case $p\in(1,2)$ and the corresponding Gagliardo-Nirenberg inequality
\be{UnScaledGN2}
\nrm{\nabla u}2^2+\nrm up^2\ge\mathsf C_p\,\nrm u2^2\quad\forall\,u\in\mathrm H^1(\R^d)\cap\mathrm L^p(\R^d)
\ee
where, compared to the case $p>2$, the positions of the norms $\nrm u2^2$ and $\nrm up^2$ have been exchanged. A scaling similar to the one of~\eqref{ScaledGN} shows that, for any $\lambda>0$,
\be{ScaledGN2}
\nrm{\nabla u}2^2+\lambda^{2+d\,\frac{2-p}p}\,\nrm up^2\ge\mathsf C_p\,\lambda^2\,\nrm u2^2\quad\forall\,u\in\mathrm H^1(\R^d)\cap\mathrm L^p(\R^d)\quad\forall\,\lambda>0\,.
\ee
By optimizing on $\lambda>0$, we obtain the scale-invariant inequality
\[
\nrm{\nabla u}2^\frac{d\,(2-p)}{d\,(2-p)+2\,p}\,\nrm u p^\frac{2\,p}{d\,(2-p)+2\,p}\ge\mathsf S_p^{1/2}\,\nrm u 2\quad\forall\,u\in\mathrm H^1(\R^d)\cap\mathrm L^p(\R^d)
\]
with
\[\label{defsp2}
\mathsf S_p=\tfrac1{d\,(2-p)+2\,p}\(2\,p\)^\frac{2\,p}{d\,(2-p)+2\,p}\,\(d\,(2-p)\)^\frac {d\,(2-p)}{d\,(2-p)+2\,p}\,\mathsf C_p\,.
\]
Optimal functions for~\eqref{UnScaledGN2} or~\eqref{ScaledGN2} have compact support according to, \emph{e.g.},~\cite{MR1938658,MR1387457,MR1629650,MR1715341}. See Section~\ref{Sec:EL} for more details.

The logarithmic Sobolev inequality corresponds to the limit case $p=2$. Let us consider~\eqref{ScaledGN} written with $\lambda^2=\frac1{p-2}$, \emph{i.e.},
\[
\nrm{\nabla\psi}2^2-\tfrac1{p-2}\(\nrm\psi p^2-\nrm\psi2^2\)\ge\,\left[\mathsf C_p\(\tfrac1{p-2}\)^{1-d\,\frac{p-2}{2\,p}}-\tfrac1{p-2}\right]\,\nrm\psi p^2\,.
\]
By passing to the limit as $p\to2$, we recover the Euclidean logarithmic Sobolev inequality with optimal constant in case $\gamma=1/2$. The general case corresponding to any $\gamma>0$, that is
\be{LS}
\nrm{\nabla\psi}2^2\ge\gamma\ird{\psi^2\,\log\(\frac{\psi^2}{\nrm\psi2^2}\)}+\tfrac d2\,\gamma\,\log\big(\tfrac{\pi\,e^2}\gamma\big)\,\nrm\psi2^2\quad\forall\,\psi\in\mathrm H^1(\R^d)\,,
\ee
follows by a simple scaling argument. It was proved in~\cite{MR1132315} that there is equality in the above inequality if and only if, up to a translation and a multiplication by a constant, $\psi(x)=e^{-\gamma\,|x|^2/4}$.

As a consequence, we obtain that the limit of $C_p$ as $p\to2_+$ is $1$ and
\[\label{LimLS}
\lim_{p\to2_+}\left[\mathsf C_p\(\tfrac1{p-2}\)^{1-d\,\frac{p-2}{2\,p}}-\tfrac1{p-2}\right]=\tfrac d4\,\log\big(\pi\,e^2\big)\,.
\]
In other words, this means that
\[
\mathsf C_p=1-\tfrac d{2\,p}\,(p-2)\,\log(p-2)+\,\tfrac d4\,\log\big(\pi\,e^2\big)\,(p-2)+o(p-2)\quad\mbox{as}\quad p\to2_+\,.
\]
Let $\varepsilon=p-2\to0_+$. We have shown that
\be{Cp2}
\mathsf C_p=1-\,\tfrac d4\,\varepsilon\,\log\varepsilon+\,\tfrac d4\,\varepsilon\,\log\big(\pi\,e^2\big)+o(\varepsilon)\,.
\ee

%%%%%%%%%%%%%%%%%%%%%%%%%%%%%%%%%%%%%%%%%%%%%%%%%%%%%%%%%%%%%%%%%%%%%%
\subsection{Case \texorpdfstring{$p\in(2,+\infty)$}{p in(2,infty)}}
Let
\[
\mu_{\mathrm{interp}}(\alpha):=\left\{\begin{array}{ll}
\mathsf S_p\,(\alpha+\Lambda)\,\Lambda^{-d\,\frac{p-2}{2\,p}}\quad&\mbox{if}\quad\alpha\in\left[-\Lambda,\tfrac{\Lambda\,(2\,p-d\,(p-2))}{d\,(p-2)}\right]\,,\\
\mathsf C_p\,\alpha^{1-d\,\frac{p-2}{2\,p}}\quad&\mbox{if}\quad\alpha\ge\tfrac{\Lambda\,(2\,p-d\,(p-2))}{d\,(p-2)}\,,
\end{array}\right.
\]
where $\mathsf C_p$ denotes the optimal constant in~\eqref{UnScaledGN} and $\mathsf S_p$ is given by~\eqref{defsp1}.
%---------------------------------------------------------------------
\begin{prop}\label{T1-gen} Let $d=2$ or $3$. Consider a magnetic field $\mathbf B$ with magnetic potential $\mathbf A$ and assume that~\eqref{Gap} holds for some $\Lambda=\Lambda[\mathbf B]>0$. For any $p\in(2,+\infty)$, any $\alpha>-\Lambda$, the function $\mu_{\mathbf B}(\alpha)$ defined in~\eqref{Interp1} satisfies
\[
\mu_{\mathbf B}(\alpha)\ge\mu_{\mathrm{interp}}(\alpha)\,.
\]\end{prop}
%---------------------------------------------------------------------
\begin{proof} Let $t\in[0,1]$. From the diamagnetic inequality
\be{diamagnetic}
\nrm{\nabla|\psi|}2\le\nrm{\nabla_{\!\mathbf A}\psi}2
\ee
and from~\eqref{Gap} and~\eqref{ScaledGN} applied with $\lambda=\frac{\alpha+\Lambda\,t}{1-t}$, we deduce that
\begin{multline*}
\nrm{\nabla_{\!\mathbf A}\psi}2^2+\alpha\,\nrm\psi2^2\ge t\,\(\nrm{\nabla_{\!\mathbf A}\psi}2^2-\Lambda\,\nrm\psi2^2\)+(1-t)\,\(\nrm{\nabla|\psi|}2+\frac{\alpha+\Lambda\,t}{1-t}\,\nrm\psi2^2\)\\
\ge\mathsf C_p\,(1-t)^{\frac{d\,(p-2)}{2\,p}}\,(\alpha+t\,\Lambda)^{1-d\,\frac{p-2}{2\,p}}\,\nrm\psi p^2
\end{multline*}
for any $\psi\in\mathrm H^1_{\!\mathbf A}$. Finally we can optimize the quantity
\[
t\mapsto (1-t)^{\frac{d\,(p-2)}{2\,p}}\,(\alpha+t\,\Lambda)^{1-d\,\frac{p-2}{2\,p}}
\]
on the interval $t\in[\max\{0,-\alpha/\Lambda\},1]$. The optimum value in the interval $\(-\alpha/\Lambda,1\)$ is achieved for $t={1-d\,\frac{p-2}{2\,p}}-\frac{d\,\alpha\,(p-2)}{2\,\Lambda\,p}$, which proves the first inequality. For $\alpha\ge\tfrac{\Lambda\,(2\,p-d\,(p-2))}{d\,(p-2)}$, the maximum is achieved at $t=0$, which proves the second inequality.\end{proof}

By duality the estimates of Proposition~\ref{T1-gen} provide a lower estimate for the best constant in the Keller-Lieb-Thirring estimate~\eqref{KLT1}.
%---------------------------------------------------------------------
\begin{cor}\label{Cor:KLTT1} Under the assumptions of Proposition~\ref{T1-gen}, for any $q=p/(p-2)\in(d/2,+\infty)$ and any potential $V$ such that in $V_-\in\mathrm L^q(\R^d)$, we have
\[\begin{array}{ll}
\lambda_{\mathbf A,V}\ge\Lambda -\mathsf S_p^{-1}\,\Lambda^{\frac{d}{2\,q}}\,\nrm {V_-}q\quad&\mbox{if}\quad\nrm {V_-}q\in\left[0,\tfrac{2\,q}{d}\,\Lambda^{1-\frac{d}{2\,q}}\,\mathsf S_p\right]\,,\\
\lambda_{\mathbf A,V}\ge -\(\mathsf C_p^{-1}\,\nrm {V_-}q\)^\frac {2\,q}{2\,q-d}\quad&\mbox{if}\quad\nrm {V_-}q\ge\tfrac{2\,q}{d}\,\Lambda^{1-\frac{d}{2\,q}}\,\mathsf S_p\,.
\end{array}\]
\end{cor}
%---------------------------------------------------------------------
\begin{proof} With $p=\frac{2\,q}{q-1}$, the estimates of Proposition~\ref{T1-gen} on $\alpha\mapsto\mu_{\mathbf B}(\alpha)$ provide estimates on its inverse $\mu\mapsto\alpha_{\mathbf B}(\mu)$ which go as follows:
\[\begin{array}{ll}
\alpha_{\mathbf B}(\mu)\le\mathsf S_p^{-1}\,\Lambda^{\tfrac{d}{2\,q}}\,\mu-\Lambda\quad&\mbox{if}\quad\mu\in\left[0,\tfrac{2\,q}{d}\,\Lambda^{1-\frac{d}{2\,q}}\,\mathsf S_p\right]\,,\\
\alpha_{\mathbf B}(\mu)\le\(\mathsf C_p^{-1}\,\mu\)^{\tfrac{2\,q}{2\,q-d}}\quad&\mbox{if}\quad\mu\ge\tfrac{2\,q}{d}\,\Lambda^{1-\frac{d}{2\,q}}\,\mathsf S_p\,.
\end{array}\]
The result is then a consequence of Corollary~\ref{Cor:KLT1}.\end{proof}

%%%%%%%%%%%%%%%%%%%%%%%%%%%%%%%%%%%%%%%%%%%%%%%%%%%%%%%%%%%%%%%%%%%%%%
\subsection{Further interpolation inequalities in case \texorpdfstring{$p\in(2,+\infty)$}{p in(2,infty)}}
Without magnetic field, Gagliardo-Nirenberg interpolation inequalities can be put in scale-invariant form~\eqref{GN} by optimizing~\eqref{ScaledGN} on the scale parameter \hbox{$\lambda>0$}. In the presence of a magnetic field, one may wonder if an inequality similar to~\eqref{GN} exists. The following result provides a positive answer.%---------------------------------------------------------------------
\begin{cor}\label{GNmag} Under the assumptions of~Proposition~\ref{T1-gen}, with $\Lambda=\Lambda[\mathbf B]$, for any $\theta\in[1-2/p,1)$ and any $\psi\in\mathrm H^1_{\!\mathbf A}(\R^d)$, we have
\begin{multline*}
\(\nrm{\nabla_{\!\mathbf A}\psi}2^2+\alpha\,\nrm\psi2^2\)^{\theta/2}\,\nrm\psi2^{1-\theta}\\
\ge\mu_{\mathrm{interp}}(\alpha)^{\frac14\,(p\,\theta-p+2)}\,\Big(\min\big\{1, (1+\tfrac\alpha\Lambda)^{1-\frac2p}\big\}\,\mathsf S_p\Big)^{\frac p4\,(1-\theta)}\,\nrm\psi p\,.
\end{multline*}\end{cor}
%---------------------------------------------------------------------
\begin{proof} With $\theta_\star=1-\frac2p$, we can write
\begin{multline*}
\(\nrm{\nabla_{\!\mathbf A}\psi}2^2+\alpha\,\nrm\psi2^2\)^{\theta/2}\,\nrm\psi2^{1-\theta}\\
=\(\nrm{\nabla_{\!\mathbf A}\psi}2^2+\alpha\,\nrm\psi2^2\)^{\frac12\,\frac{\theta-\theta_\star}{1-\theta_\star}}
\(\(\nrm{\nabla_{\!\mathbf A}\psi}2^2+\alpha\,\nrm\psi2^2\)^{\frac12\(1-\frac2p\)}\nrm\psi2^{\frac2p}\)^{\frac{1-\theta}{1-\theta_\star}}\\
\ge\(\mu_{\mathrm{interp}}(\alpha)\,\nrm\psi p^2\)^{\frac12\,\frac{\theta-\theta_\star}{1-\theta_\star}}\(\(\nrm{\nabla_{\!\mathbf A}\psi}2^2+\alpha\,\nrm\psi2^2\)^{\frac12\(1-\frac2p\)}\nrm\psi2^{\frac2p}\)^{\frac{1-\theta}{1-\theta_\star}}\,.
\end{multline*}
If $\alpha\in(-\Lambda,0]$, it follows from~\eqref{Gap} and~\eqref{diamagnetic} that
\[
\nrm{\nabla_{\!\mathbf A}\psi}2^2+\alpha\,\nrm\psi2^2\ge\big(1+\tfrac\alpha\Lambda\big)\,\nrm{\nabla_{\!\mathbf A}\psi}2^2\ge\big(1+\tfrac\alpha\Lambda\big)\,\nrm{\nabla|\psi|}2^2\,,
\]
while we can simply drop $\alpha\,\nrm\psi2^2$ when $\alpha\ge0$. Hence it follows from~\eqref{GN} that
\begin{multline*}
\(\nrm{\nabla_{\!\mathbf A}\psi}2^2+\alpha\,\nrm\psi2^2\)^{\theta/2}\,\nrm\psi2^{1-\theta}\\
\ge\(\mu_{\mathrm{interp}}(\alpha)\,\nrm\psi p^2\)^{\frac12\,\frac{\theta-\theta_\star}{1-\theta_\star}}\(\min\left\{1,(1+\tfrac\alpha\Lambda)^{\frac12\(1-\frac2p\)}\right\}\,
\mathsf S_p^{1/2}\,\nrm\psi p\)^\frac{1-\theta}{1-\theta_\star}\,,
\end{multline*}
which concludes the proof.\end{proof}

%%%%%%%%%%%%%%%%%%%%%%%%%%%%%%%%%%%%%%%%%%%%%%%%%%%%%%%%%%%%%%%%%%%%%%
\subsection{Case \texorpdfstring{$p\in(\pmin,2)$}{p in(\pmin,2)}}
Let $\nu_{\mathrm{interp}}$ be given by
\[
\nu_{\mathrm{interp}}(\beta):=\left\{\begin{array}{lr}
\mathsf C_p\,\beta^\frac{2\,p}{2\,p+d\,(2-p)}&\mbox{if}\quad\beta\ge\beta_\star:=\(\tfrac{2\,p+d\,(2-p)}{d\,(2-p)}\,\Lambda\,\mathsf C_p^{-1}\)^\frac{2\,p+d\,(2-p)}{2\,p}\,,\\
\Lambda+\beta\,\Lambda^\frac{d\,(p-2)}{2\,p}\,\tfrac{2\,p}{d\,(2-p)}\,\(\tfrac{d\,(2-p)}{2\,p+d\,(2-p)}\)^\frac{2\,p+d\,(2-p)}{2\,p}\,\mathsf C_p^\frac{2\,p+d\,(2-p)}{2\,p}\hspace*{-116pt}&\mbox{if}\quad\beta\in[0,\beta_\star]\,,
\end{array}\right.
\]
where $\mathsf C_p$ denotes the optimal constant in~\eqref{UnScaledGN2}.
%---------------------------------------------------------------------
\begin{prop}\label{T1-gen-qsmall} Let $d=2$ or $3$. Consider a magnetic field $\mathbf B$ with magnetic potential $\mathbf A$ and assume that~\eqref{Gap} holds for some $\Lambda=\Lambda[\mathbf B]>0$. For any $p\in(\pmin,2)$, any $\beta>0$, the function $\nu_{\mathbf B}$ defined in~\eqref{Interp2} satisfies
\[
\nu_{\mathbf B}(\beta)\ge\nu_{\mathrm{interp}}(\beta)\,.
\]\end{prop}
%---------------------------------------------------------------------
\begin{proof} For all $\psi\in\mathrm H^1_{\!\mathbf A}$, by~\eqref{Gap} and~\eqref{diamagnetic}, we obtain that
\begin{multline*}
\nrm{\nabla_{\!\mathbf A}\psi}2^2+\beta\,\nrm\psi p^2=t\,\(\nrm{\nabla_{\!\mathbf A}\psi}2^2-\Lambda\,\nrm\psi2^2\)+(1-t)\,\nrm{\nabla_{\!\mathbf A}\psi}2^2+\beta\,\nrm\psi p^2+\Lambda\,t\,\nrm\psi2^2\\\ge(1-t)\,\nrm{\nabla|\psi|}2^2+\beta\,\nrm\psi p^2+\Lambda\,t\,\nrm\psi2^2\,.
\end{multline*}
Next we apply~\eqref{ScaledGN2} to $u=|\psi|$ with $\lambda^2=\(\frac\beta{1-t}\)^\frac{2\,p}{2\,p+d\,(2-p)}$. This yields
\[
\nrm{\nabla_{\!\mathbf A}\psi}2^2+\beta\,\nrm\psi p^2\ge\left[(1-t)^\frac{d\,(2-p)}{2\,p+d\,(2-p)}\,\beta^\frac{2\,p}{2\,p+d\,(2-p)}\,\mathsf C_p+\Lambda\,t\right]\,\nrm\psi2^2\,.
\]
If $\beta\le\beta_\star$, the right hand side is maximal for some explicit $t\in[0,1]$, otherwise the maximum on $[0,1]$ is achieved by $t=0$, which concludes the proof.\end{proof}

By duality the estimates of Proposition~\ref{T1-gen-qsmall} provide a lower estimate for the best constant in the Keller-Lieb-Thirring estimate~\eqref{KLT2}.
%---------------------------------------------------------------------
\begin{cor} Under the assumptions of Proposition~\ref{T1-gen-qsmall}, for any $q=p/(2-p)\in(1,+\infty)$ and any nonnegative potential $W$ such that $W^{-1}\in\mathrm L^q(\R^d)$, we have
\begin{eqnarray*}
&&\lambda_{\mathbf A, W}\ge\nu_{\mathbf B}\(\nrm{W^{-1}}q^{-1}\)\ge\Lambda+\Lambda^{\tfrac{d\,(p-2)}{2\,p}}\,\tfrac{2\,p}{d\,(2-p)}\,\(\tfrac{d\,(2-p)}{2\,p+d\,(2-p)}\,\mathsf C_p\)^\frac{2\,p+d\,(2-p)}{2\,p}\,\nrm{W^{-1}}q^{-1}\\
&&\hspace*{7.8cm}\mbox{if}\quad\nrm{W^{-1}}q^{-1}\in[0,\beta_\star]\,,\\
&&\lambda_{\mathbf A, W}\ge\nu_{\mathbf B}\(\nrm{W^{-1}}q^{-1}\)\ge\mathsf C_p\,\nrm{W^{-1}}q^\frac{-\,2\,p}{2\,p+d\,(2-p)}\quad\mbox{if}\quad\nrm{W^{-1}}q^{-1}\ge\beta_\star\,.
\end{eqnarray*}
\end{cor}
%---------------------------------------------------------------------

%%%%%%%%%%%%%%%%%%%%%%%%%%%%%%%%%%%%%%%%%%%%%%%%%%%%%%%%%%%%%%%%%%%%%%
\subsection{Proof of \texorpdfstring{Theorem~\ref{Thm:Interp}}{Theorem 1.1}}\label{ProofOfTheorem1}
\begin{proof}[Proof of Theorem~\ref{Thm:Interp}] Let us consider Case (i): $p\in(2,2^*)$. The positivity of the function $\mu_{\mathbf B}$ is a consequence of Proposition~\ref{T1-gen} while the concavity follows from the definition of $\alpha\mapsto\mu_{\mathbf B}(\alpha)$ as the infimum on $\mathrm H^1_{\!\mathbf A}(\R^d)$ of an affine function of $\alpha$. The estimate as $\alpha\to(-\Lambda)_+$ is easily obtained by considering as test function the function $\psi\in\mathrm H^1_{\mathbf A}(\R^d)$ for which there is equality in~\eqref{Gap}. We know from Proposition~\ref{T1-gen} that
\[
\lim_{\alpha\to+\infty}\mu_{\mathbf B}(\alpha)\,\alpha^{\frac{d-2}2-\frac dp}\ge\mathsf C_p\,.
\]
To prove the equality, we take as test function for $\mu_{\mathbf B}(\alpha)$ the function $v_\alpha:=v(\sqrt{\alpha}\,\cdot)$, with $\alpha>0$, where the radial function $v$ realizes the equality in~\eqref{UnScaledGN}. The function $v$ is smooth, positive everywhere and decays like $e^{-|x|}$ as $|x|\to+\infty$. Notice that $v_\alpha$ realizes the equality in~\eqref{ScaledGN} and there is a constant $C>0$ such that \hbox{$v_\alpha(x)\le C\,\exp\big(-\sqrt\alpha\,|x|\big)$} for any $x\in\R^d$. Since $\nrm{\nabla_{\!\mathbf A}v}2^2\le\nrm{\nabla v}2^2+2\,\nrm{\nabla v}2\,\nrm{\mathbf A\,v}2+\nrm{\mathbf A\,v}2^2$, we obtain that
\[
\frac{\nrm{\nabla_{\!\mathbf A}v_\alpha}2^2+\alpha\,\nrm{v_\alpha}2^2}{\alpha^{\frac{2-d}2+\frac dp}\,\nrm{v_\alpha}p^2}\le\mathsf C_p+2\,\sqrt{\mathsf C_p}\,\varepsilon+\varepsilon^2\quad\mbox{with}\quad\varepsilon^2=C^2\,\frac{\ird{|\mathbf A(x)|^2\,e^{-\,2\,\sqrt\alpha\,|x|}}}{\alpha^{\frac{2-d}2+\frac dp}\,\nrm{v_\alpha}p^2}\,.
\]
The result follows from $\alpha^{\frac{2-d}2+\frac dp}\,\nrm{v_\alpha}p^2=\alpha^\frac{2-d}2\,\nrm vp^2$ and~\eqref{hyppp} with $\sigma=2\sqrt\alpha$.

The proof of (ii) is very similar to that of (i). The positivity of the function $\nu_{\mathbf B}$ is a consequence of Proposition~\ref{T1-gen-qsmall} while the concavity follows from the definition of $\beta\mapsto\nu_{\mathbf B}(\beta)$. From Proposition~\ref{T1-gen-qsmall}, we know that
\[
\lim_{\beta\to+\infty}\nu_{\mathbf B}(\beta)\,\beta^{-\frac{2\,p}{2\,p+d\,(2-p)}}\ge\mathsf C_p\,.
\]
To prove the equality, for any $\beta>0$, we take as test function for $\nu_{\mathbf B}(\beta)$ the function
\[
w_\beta(x):=w\(\beta^\frac p{2\,p+d\,(2-p)}\,x\)\quad\forall\,x\in\R^d\,,
\]
where the radial function $w$ realizes the equality in~\eqref{UnScaledGN2}, so that $w_\beta$ realizes the equality in~\eqref{ScaledGN2}. The function $w$ has compact support and can be estimated from above and from below, up to a multiplicative constant, by the characteristic function of centered balls. The same computation as above shows that
\[
\frac{\nrm{\nabla_{\!\mathbf A}w_\beta}2^2+\beta\,\nrm{w_\beta}p^2}{\beta^\frac{2\,p}{2\,p+d\,(2-p)}\,\nrm{w_\beta}2^2}\le\mathsf C_p+2\,\sqrt{\mathsf C_p}\,\varepsilon+\varepsilon^2
\]
with $\varepsilon^2=C^2\,\frac{\ird{|\mathbf A(x)|^2\,\left|w\big(\beta^\frac p{2\,p+d\,(2-p)}\,x\big)\right|^2}}{\beta^\frac{2\,p}{2\,p+d\,(2-p)}\,\nrm{w_\beta}2^2}$. The result follows from
\[
\beta^\frac{2\,p}{2\,p+d\,(2-p)}\,\nrm{w_\beta}2^2=\beta^\frac{(2-d)\,p}{2\,p+d\,(2-p)}\,\nrm w2^2
\]
and~\eqref{hyppp} with $\sigma=\beta^\frac p{2\,p+d\,(2-p)}$.

The case $p=2$ is much simpler. As a straightforward consequence of the Euclidean logarithmic Sobolev inequality~\eqref{LS} and of the diamagnetic inequality~\eqref{diamagnetic}, we know that
\[
\nrm{\nabla_{\!\mathbf A}\psi}2^2\ge\gamma\ird{|\psi|^2\,\log\(\frac{|\psi|^2}{\nrm\psi2^2}\)}+\tfrac d2\,\gamma\,\log\big(\tfrac{\pi\,e^2}\gamma\big)\,\nrm\psi2^2\quad\forall\,\psi\in\mathrm H^1_{\!\mathbf A}(\R^d)\,.
\]
As a consequence, we deduce the existence of a concave function $\xi_{\mathbf B}$ in Inequality~\eqref{KLT3}, such that
\[
\xi_{\mathbf B}\(\gamma\)\ge\tfrac d2\,\gamma\,\log\big(\tfrac{\pi\,e^2}\gamma\big)\quad\forall\,\gamma>0\,.
\]
Note that the r.h.s.~is negative for $\gamma$ large. The function $w_\gamma(x)=(\gamma/\pi)^{d/4}\,e^{-\frac\gamma2\,|x|^2}$ is optimal in~\eqref{LS} and can be used as a test function in~\eqref{Interp3} in the regime as $\gamma\to+\infty$. Using the fact that $\nrm{w_\gamma}2=1$, $\nrm{\nabla w_\gamma}2=\sqrt{d\,\gamma}$ and
\begin{eqnarray*}
\nrm{\nabla_{\!\mathbf A}w_\gamma}2^2&\kern-8pt\le&\kern-8pt\nrm{\nabla w_\gamma}2^2+2\,\nrm{\nabla w_\gamma}2\,\nrm{A\,w_\gamma}2+\nrm{A\,w_\gamma}2^2\\
&\kern-8pt=&\kern-8pt\gamma\,\ird{|w_\gamma|^2\,\log|w_\gamma|^2}+\tfrac d2\,\gamma\,\log\big(\tfrac{\pi\,e^2}\gamma\big)+2\,\nrm{\nabla w_\gamma}2\,\nrm{A\,w_\gamma}2+\nrm{A\,w_\gamma}2^2\,,
\end{eqnarray*}
we get that, for some positive constant $c$,
\begin{multline*}
0\le\nrm{\nabla_{\!\mathbf A}w_\gamma}2^2-\gamma\,\ird{|w_\gamma|^2\,\log|w_\gamma|^2}-\xi_{\mathbf B}\(\gamma\)\\
\le\tfrac d2\,\gamma\,\log\big(\tfrac{\pi\,e^2}\gamma\big)-\xi_{\mathbf B}\(\gamma\)+2\,\sqrt{d\,\gamma}\,\nrm{A\,w_\gamma}2+\nrm{A\,w_\gamma}2^2\\
\le\tfrac d2\,\gamma\,\log\big(\tfrac{\pi\,e^2}\gamma\big)\left[1-\frac{\xi_{\mathbf B}\(\gamma\)}{\tfrac d2\,\gamma\,\log\big(\tfrac{\pi\,e^2}\gamma\big)}-\,\frac{c\,\varepsilon}{\sqrt{\log\big(\tfrac\gamma{\pi\,e^2}\big)}}-\varepsilon^2\right]
\end{multline*}
where $\varepsilon^2=\frac{\gamma^{\frac d2-1}\ird{|\mathbf A(x)|^2\,e^{-\,\gamma\,|x|^2}}}{\tfrac d2\,\log\big(\tfrac\gamma{\pi\,e^2}\big)\,\pi^\frac d2}\to0$ as $\gamma\to+\infty$ according to~\eqref{hyppp}. This establishes that $\xi_{\mathbf B}\(\gamma\)$ is equal to $\tfrac d2\,\gamma\,\log\big(\pi\,e^2/\gamma\big)$ at leading order as $\gamma\to+\infty$.\end{proof}

%%%%%%%%%%%%%%%%%%%%%%%%%%%%%%%%%%%%%%%%%%%%%%%%%%%%%%%%%%%%%%%%%%%%%%
%%%%%%%%%%%%%%%%%%%%%%%%%%%%%%%%%%%%%%%%%%%%%%%%%%%%%%%%%%%%%%%%%%%%%%
\section{Lower estimates: constant magnetic field in dimension two}\label{constant-magnetic-ineq}

In the particular case when the magnetic field is constant, of strength $B>0$, and $d=2$, we can improve the lower estimates of the last section. In this section we assume that $\mathbf B=(0,B)$ and choose the gauge so that
\be{B-d=2}
\mathbf A_1=\tfrac B2x_2\,,\quad\mathbf A_2=-\tfrac B2x_1\quad\forall\,x=(x_1,x_2)\in\R^2\,.
\ee

%%%%%%%%%%%%%%%%%%%%%%%%%%%%%%%%%%%%%%%%%%%%%%%%%%%%%%%%%%%%%%%%%%%%%%
\subsection{A preliminary result}
The next result follows from~\cite[proof of Theorem~3.1]{MR1462758} by Loss and Thaller.
%---------------------------------------------------------------------
\begin{prop}\label{Prop:Loss-Thaller} Consider a constant magnetic field with field strength $B$ in two dimensions. For every $c\in[0,1]$, we have
\[
\irtwo{|\nabla_{\!\mathbf A}\psi|^2}\ge\(1-c^2\)\irtwo{|\nabla\psi|^2}+c\,B\irtwo{\psi^2}\,,
\]
and equality holds with $\psi=u\,e^{iS}$ and $u>0$ if and only if
\be{equlaitycaseLT}
\(-\partial_2u^2,\,\partial_1u^2\)=\frac{2\,u^2}c\,(\mathbf A+\nabla S)\,.
\ee
\end{prop}
%---------------------------------------------------------------------
\begin{proof} For every $c\in[0,1]$,
\begin{multline*}
\irtwo{|\nabla_{\!\mathbf A}\psi|^2}=\irtwo{|\nabla u|^2}+\irtwo{|\mathbf A+\nabla S|^2\,u^2}\\
=\(1-c^2\)\irtwo{|\nabla u|^2}+\irtwo{\(c^2\,|\nabla u|^2+|\mathbf A+\nabla S|^2\,u^2\)}\,.
\end{multline*}
An expansion of the square shows that
\[
\irtwo{\(c^2\,|\nabla u|^2+|\mathbf A+\nabla S|^2\,u^2\)}\ge\irtwo{2\,c\,|\nabla u|\,|\mathbf A+\nabla S|\,u}\,,
\]
with equality only if $c\,|\nabla u|=|\mathbf A+\nabla S|\,u$. Next we obtain that
\[
2\,|\nabla u|\,|\mathbf A+\nabla S|\,u=|\nabla u^2|\,|\mathbf A+\nabla S|\ge\(\nabla u^2\)^\perp\cdot(\mathbf A+\nabla S)\,,
\]
where $\(\nabla u^2\)^\perp:=\(-\partial_2u^2,\,\partial_1u^2\)$, and there is equality if and only if
\[
\(-\partial_2u^2,\,\partial_1u^2\)=\gamma\,(\mathbf A+\nabla S)
\]
for some $\gamma$. Since $c\,|\nabla u|=|\mathbf A+\nabla S|\,u$, we have $\gamma=2\,u^2/c$. Integration by parts yields
\[
\irtwo{\(c^2\,|\nabla u|^2+|\mathbf A+\nabla S|^2\,u^2\)}\ge B\,c\irtwo{u^2}\,.
\]
\end{proof}

%%%%%%%%%%%%%%%%%%%%%%%%%%%%%%%%%%%%%%%%%%%%%%%%%%%%%%%%%%%%%%%%%%%%%%
\subsection{Case \texorpdfstring{$p\in(2,+\infty)$}{p in(2,infty)}}
%---------------------------------------------------------------------
\begin{prop}\label{T1-gen-bis} Consider a constant magnetic field with field strength $B$ in two dimensions. Given any $p\in(2,+\infty)$, and any $\alpha>-B$, we have
\be{estt1}
\mu_{\mathbf B}(\alpha)\ge\mathsf C_p\(1-c^2\)^{1-\frac2p}\,(\alpha+c\,B)^{\frac2p}=:\mu_{\kern1pt\mathrm{LT}}(\alpha)\,,
\ee
with
\be{cpeta}
c=c(p,\eta)=\frac{\sqrt{\eta^2+p-1}-\eta}{p-1}=\frac1{\eta+\sqrt{\eta^2+p-1}}\in(0,1)
\ee
and $\eta=\alpha\,(p-2)/(2\,B)$.\end{prop}
%---------------------------------------------------------------------
\begin{proof} For any $\alpha>-B$, $\psi\in\mathrm H^1_{\!\mathbf A}(\R^2)$ and $c\in[0,1]$ such that $\alpha+c\,B\ge0$, we use Proposition~\ref{Prop:Loss-Thaller} to write
\[
\nrm{\nabla_{\!\mathbf A}\psi}2^2+\alpha\,\nrm\psi2^2\ge\(1-c^2\)\irtwo{|\nabla u|^2}+(\alpha+c\,B)\irtwo{u^2}
\]
with $u=|\psi|$. By applying~\eqref{ScaledGN} with $\lambda^2=(\alpha+c\,B)/\big(1-c^2\big)$, we get
\[
\nrm{\nabla_{\!\mathbf A}\psi}2^2+\alpha\,\nrm\psi2^2\ge\mathsf C_p\(1-c^2\)^{1-\frac2p}\,(\alpha+c\,B)^{\frac2p}\,\nrm\psi p^2\,.
\]
Next we optimize the function $c\mapsto\(1-c^2\)^{1-\frac2p}\,(\alpha+c\,B)^{\frac2p}$ in the interval $[0,1]$. This function reaches its maximum at $c$ such that
\[
(p-2)\,c\,(\alpha+c\,B)=B\,(1-c^2)\,.
\]
Notice that $\alpha+c\,B$ is nonnegative. With
\[
\eta=\frac{\alpha\,(p-2)}{2\,B}\,,
\]
the equation for $c$ becomes
\[
(p-1)\,c^2+\,2\,\eta\,c-1=0\,.
\]
which is solved by~\eqref{cpeta}.\end{proof}

%%%%%%%%%%%%%%%%%%%%%%%%%%%%%%%%%%%%%%%%%%%%%%%%%%%%%%%%%%%%%%%%%%%%%%
\subsection{Case \texorpdfstring{$p\in(\pmin,2)$}{p in(\pmin,2)}}
Now let us turn our attention to the case $p\in(\pmin,2)$. The strategy of the proof of Proposition~\ref{T1-gen-bis} applies: for any $c\in(0,1)$, for any $\beta>0$, by applying~\eqref{ScaledGN2} with $\lambda^{4/p}=\beta/(1-c^2)$, we obtain
\[
\nrm{\nabla_{\!\mathbf A}\psi}2^2+\beta\,\nrm\psi p^2\ge\(c\,B+\mathsf C_p\,\beta^\frac p2\,(1-c^2)^{1-\frac p2}\)\,\nrm\psi2^2\,.
\]
The function $c\mapsto c\,B+\mathsf C_p\,\beta^{p/2}\,(1-c^2)^{1-p/2}$ is positive in $[0,1]$ and its derivative is positive at $0_+$, and negative in a neighborhood of $1_-$. The maximum is achieved at the unique point $c_\ast\in(0,1)$ given by
\be{cstar}
\frac {c_\ast}{(1-c_\ast^2)^{p/2}}=\frac B{(2-p)\,\mathsf C_p\,\beta^{p/2}}\,.
\ee
This establishes the following result.
%---------------------------------------------------------------------
\begin{prop}\label{T1-gen-pless23} Consider a constant magnetic field with field strength $B$ in two dimensions. Given any $p\in(\pmin,2)$, and any $\beta>0$, we have
\[\label{estt1pless2}
\nu_{\mathbf B}(\beta)\ge c_\ast\,B+\mathsf C_p\,\beta^\frac p2\,(1-c_\ast^2)^{1-\frac p2}=:\nu_{\kern1pt\mathrm{LT}}(\beta)
\]
with $c_\ast$ given by~\eqref{cstar}.\end{prop}
%---------------------------------------------------------------------

%%%%%%%%%%%%%%%%%%%%%%%%%%%%%%%%%%%%%%%%%%%%%%%%%%%%%%%%%%%%%%%%%%%%%%
\subsection{Logarithmic Sobolev inequality}
By passing to the limit as $p\to2_+$ in~\eqref{estt1}, we obtain a two-dimensional magnetic logarithmic Sobolev inequality.
%---------------------------------------------------------------------
\begin{lem}\label{Lem:LS} Consider a constant magnetic field with field strength $B>0$ in two dimensions. Then for any $\gamma>0$, the best constant in~\eqref{Interp3} satisfies
\be{LS4a}
\xi_{\mathbf B}(\gamma)\ge B\,c\(2,\,\tfrac\gamma B\)+\gamma\,\log\(\tfrac{\pi\,e^2\,c\(2,\,\gamma/B\)}B\)\,,
\ee
where $c(2,\eta):=\sqrt{\eta^2+1}-\eta$.\end{lem}
%---------------------------------------------------------------------
\begin{proof} By using~\eqref{Cp2} with $d=2$ and~\eqref{cpeta}, we see that for any $\eta>0$,
\begin{multline*}
\mathsf C_p\,(1-c^2)^{1-\frac2p}\,\big(\tfrac{2\,\eta\,B}{p-2}+c\,B\big)^{\frac2p}-\tfrac{2\,\eta\,B}{p-2}\\
=\tfrac{2\,\eta\,B}\varepsilon\,\Big[\Big(1-\,\tfrac\varepsilon2\,\log\varepsilon+\,\tfrac\varepsilon2\,\log\big(\pi\,e^2\big)\Big)\(1+\,\tfrac\varepsilon2\,\log\(1-c^2\)\)\\
\left.\cdot\(1+\,\tfrac\varepsilon2\,\tfrac c\eta\)\(1-\,\tfrac\varepsilon2\,\log\(\tfrac{2\,\eta\,B}\varepsilon\)\)-1\right]+o(\varepsilon)\\
\to B\,\left[c(2,\eta)+\eta\,\log\(\tfrac{\pi\,e^2\,c(2,\eta)}B\)\right]
\end{multline*}
as $\varepsilon=p-2\to0_+$, because $1-c(2,\eta)^2=2\,\eta\,c(2,\eta)$. By rewriting~\eqref{estt1} with $\alpha=\tfrac{2\,\eta\,B}{p-2}$ as
\[
\nrm{\nabla_{\!\mathbf A}\psi}2^2\ge\tfrac{2\,\eta\,B}{p-2}\(\nrm\psi p^2-\nrm\psi2^2\)+\left[\mathsf C_p\,(1-c^2)^{1-\frac2p}\,\big(\tfrac{2\,\eta\,B}{p-2}+c\,B\big)^{\frac2p}-\tfrac{2\,\eta\,B}{p-2}\right]\,\nrm\psi p^2
\]
we can pass to the limit as $p\to2_+$ and establish~\eqref{LS4a} by setting $\gamma=\eta\,B$.\end{proof}

It turns out that the above magnetic logarithmic Sobolev inequality is optimal. To identify the minimizers, we observe that the magnetic Schr\"odinger operator is not invariant under the standard translations. For any $\mathbf b=(b_1,b_2)\in\R^2$,
\[
\nabla_{\!\mathbf A}\psi=(\nabla_{\!\mathbf A}\phi)(x-\mathbf b)\quad\mbox{if}\quad\phi(x-\mathbf b)=e^{-i\,B\,(b_1\,x_2-b_2\,x_1)/2}\,\psi(x)\quad\forall\,x\in\R^2
\]
and $-\Delta_{\mathbf A}$ commutes with the \emph{magnetic translations} \hbox{$\psi\mapsto e^{i\,B\,(b_1\,x_2-b_2\,x_1)/2}\,\psi(x-\mathbf b)$} if $\mathbf A$ is given by~\eqref{B-d=2}.
%---------------------------------------------------------------------
\begin{prop}\label{Prop:LS} Consider a constant magnetic field with field strength $B>0$ in two dimensions. Then the logarithmic Sobolev inequality~\eqref{Interp3} holds with
\[
\xi_{\mathbf B}(\gamma)=B\,c\(2,\,\tfrac\gamma B\)+\gamma\,\log\(\tfrac{\pi\,e^2\,c\(2,\,\gamma/B\)}B\)
\]
where $c(2,\eta):=\sqrt{\eta^2+1}-\eta$, and the optimizer is given, up to a multiplication by a complex constant and a magnetic translation, by $\psi(x)=e^{-\gamma\,|x|^2/4}$.
\end{prop}
%---------------------------------------------------------------------
In other words, optimizers in inequality~\eqref{Interp3} are of the form
\[
\psi(x)=C\,e^{-\frac\gamma4\frac{|x-\mathbf b|^2}4+\,i\,\frac B2\,(b_1\,x_2-b_2\,x_1)}\quad\forall\,x\in\R^2\,,\quad C\in\mathbb C\,,\quad\mathbf b=(b_1,b_2)\in\R^2\,.
\]
Notice that in the semi-classical regime corresponding to a limit of the magnetic field $\mathbf B$ such that $1/(2\,\eta)=\Lambda=\Lambda[\mathbf B]\to0$, we recover the classical logarithmic Sobolev inequality~\eqref{LS} without magnetic field.

\begin{proof} Using Proposition~\ref{Prop:Loss-Thaller} and Inequality~\eqref{LS}, for all $c\in [0,1]$ we obtain
\[
\irtwo{|\nabla_{\!\mathbf A}\psi|^2}\ge\sigma\,(1-c^2)\irtwo{|\psi|^2\,\log\(\tfrac{|\psi|^2}{\nrm\psi2^2}\)}+\(B\,c+\sigma\,(1-c^2)\,\log\big(\tfrac{\pi\,e^2}\sigma\big)\)\nrm\psi2^2\,.
\]
We recover~\eqref{LS4a} with $\sigma\,(1-c^2)=\gamma$ and $c=c\(2,\,\tfrac\gamma B\)$.

According to Proposition~\ref{Prop:Loss-Thaller}, equality holds if $\psi=u\,e^{iS}$ satisfies~\eqref{equlaitycaseLT} and, simultaneously, $\psi$ realizes the equality case in~\eqref{LS}, \emph{i.e.},
\[
\psi(x)=C\,e^{-\frac\gamma4\,|x-\mathbf b|^2}\quad\forall\,x\in\R^2
\]
with $C\in\mathbb C$ and $\mathbf b\in\R^2$. By~\eqref{equlaitycaseLT}, this means that $S$ has to satisfy
\[
\partial_1 S=-\frac B2\,b_2\,,\quad\partial_2 S=\frac B2\,b_1\,,
\]
which implies $S=\frac B2\,(b_1\,x_2-b_2\,x_1)+D$, for some constant $D$.\end{proof}

%%%%%%%%%%%%%%%%%%%%%%%%%%%%%%%%%%%%%%%%%%%%%%%%%%%%%%%%%%%%%%%%%%%%%%
%%%%%%%%%%%%%%%%%%%%%%%%%%%%%%%%%%%%%%%%%%%%%%%%%%%%%%%%%%%%%%%%%%%%%%
\section{An upper estimate and some numerical results}\label{Sec:Numerics}

In this section, we assume that $d=2$, consider a constant magnetic field, establish a theoretical upper bound, and numerically compute the difference with the lower bounds of Sections~\ref{general-magnetic-ineq} and~\ref{constant-magnetic-ineq}.

%%%%%%%%%%%%%%%%%%%%%%%%%%%%%%%%%%%%%%%%%%%%%%%%%%%%%%%%%%%%%%%%%%%%%%
\subsection{An upper estimate: constant magnetic field in dimension two}\label{Sec:upper}
Let $r=\sqrt{x_1^2+x_2^2}=|x|$ be the radial coordinate associated to any $x=(x_1,x_2)\in\R^2$ and assume that the magnetic potential is given by~\eqref{B-d=2}. For every integer $k\in\N$ we introduce the special symmetry class
\[
\tag{$\mathcal C_k$}
\psi(x)=\(\tfrac{x_2+\,i\,x_1}{|x|}\)^k\,v(|x|)\quad\forall\,x\in\R^2\,.
\]
With this notation, if $\psi\in\mathcal C_k$, then
\[
\frac 1{2\,\pi}\irtwo{|\nabla_{\!\mathbf A}\psi|^2}=\istwo{|v'|^2}+\istwo{\(\tfrac kr-\tfrac{B\,r}2\)^2\,|v|^2}\,.
\]
Let us define
\[
\mathcal Q_\alpha^{(p)}[\psi]:=\tfrac{\nrm{\nabla_{\!\mathbf A}\psi}2^2+\alpha\,\nrm\psi2^2}{\nrm\psi p^2}\quad\mbox{if}\quad p>2\,,\quad\mathcal Q_\beta^{(p)}[\psi]:=\tfrac{\nrm{\nabla_{\!\mathbf A}\psi}2^2+\beta\,\nrm\psi p^2}{\nrm\psi2^2}\quad\mbox{if}\quad p\in(1,2)\,.
\]
The existence of minimizers of $\mathcal Q_\alpha^{(p)}$ in $\mathcal C_k$ was proved in~\cite[Theorem~3.5]{MR1034014} for any $k\in\N$. In the class $\mathcal C_0$, with a slight abuse of notations, we have $\psi=v$ and simple upper estimates can be obtained using $v_\sigma(r)=e^{-\,r^2/(2\,\sigma)}$ as test function:
\[
\nrm{\nabla_{\!\mathbf A}v_\sigma}2^2=\tfrac\pi4\(4+\sigma^2\)\,,\quad\nrm{v_\sigma}2^2=\pi\,\sigma\quad\mbox{and}\quad\nrm{v_\sigma}p^2=\(\tfrac2p\,\pi\,\sigma\)^\frac2p\,.
\]
\emph{Case} (i). Assume first that $p\in(2,+\infty)$ and let $\theta:=2/p$. We observe that
\[
\mathcal Q_\alpha^{(p)}[v_\sigma]=\tfrac18\,(2\,\pi)^{1-\theta}\,p^\theta\,f_{\alpha,\theta}(\sigma)\quad\mbox{where}\quad f_{\alpha,\theta}(\sigma):=\sigma^{\kern1pt2-\theta}+\,4\,\alpha\,\sigma^{1-\theta}+\,4\,\sigma^{-\theta}\,.
\]
The function $f_{\alpha,\theta}$ has a unique minimum on $(0,+\infty)$, which is determined by the second order equation
\[
(2-\theta)\,\sigma^2+\,4\,\alpha\,(1-\theta)\,\sigma-\,4\,\theta=0\,,
\]
namely $\sigma=\sigma_+(\alpha,\theta)$ with
\[
\sigma_+(\alpha,\theta):=2\,\frac{\sqrt{4\,\alpha^2\,(1-\theta)^2+\theta\,(2-\theta)}-\,\alpha\,(1-\theta)}{2-\theta}\,.
\]
With $\theta=2/p$, this gives the estimate
\[
\mathcal Q_\alpha^{(p)}[v_{\sigma_+(\alpha,\theta)}]=\tfrac18\,(2\,\pi)^{1-\theta}\,p^\theta\,f_{\alpha,\theta}\big(\sigma_+(\alpha,\theta)\big)=:\mu_{\mathrm{Gauss}}(\alpha)\,.
\]
\emph{Case} (ii). When $p\in(1,2)$, with $\theta:=\frac2p\in(1,2]$ and $\kappa(\beta,\theta):=8\,\theta^\theta\,\pi^{1-\theta}\,\beta$, we get that
\[
\mathcal Q_\beta^{(p)}[v_\sigma]=\tfrac18\,g_{\beta,\theta}(\sigma)\quad\mbox{where}\quad g_{\beta,\theta}(\sigma):=\sigma+\frac2\sigma+\,\kappa(\beta,\theta)\,\sigma^{\theta-1}\,.
\]
Elementary considerations show that $g_{\beta,\theta}(\sigma)$ has a unique minimum $\sigma=\sigma_-(\beta,\theta)$ determined by the equation
\[
1-\frac2{\sigma^2}+\,\kappa(\beta,\theta)\,(\theta-1)\,\sigma^{\theta-2}=0\,,
\]
which is in general not explicit. However, a simple elimination shows that
\[
\mathcal Q_\beta^{(p)}[v_{\sigma_-(\beta,\theta)}]=\frac18\,g_{\beta,\theta}\big(\sigma_-(\beta,\theta)\big)=\frac18\(\tfrac{2\,\theta}{\theta-1}\,\tfrac1{\sigma_-(\beta,\theta)}+\tfrac{\theta-2}{\theta-1}\,\sigma_-(\beta,\theta)\)=:\nu_{\mathrm{Gauss}}(\beta)\,.
\]
%---------------------------------------------------------------------
\begin{prop}\label{Prop:Upper} With the above notations, we have
\[
\mu_{\mathbf B}(\alpha)\le\mu_{\mathrm{Gauss}}(\alpha)\quad\mbox{if}\quad p>2\quad\mbox{and}\quad
\nu_{\mathbf B}(\beta)\le\nu_{\mathrm{Gauss}}(\beta)\quad\mbox{if}\quad p\in(1,2)\,.
\]
\end{prop}
%---------------------------------------------------------------------

%%%%%%%%%%%%%%%%%%%%%%%%%%%%%%%%%%%%%%%%%%%%%%%%%%%%%%%%%%%%%%%%%%%%%%
\subsection{Numerical estimates based on Euler-Lagrange equations}\label{Sec:EL}
Instead of a Gaussian test function, one can numerically compute the minimum of $\mathcal Q_\alpha^{(p)}$ in the class $\mathcal C_0$ by solving the corresponding Euler-Lagrange equation.\\
\emph{Case} (i). Assume that $p\in(2,+\infty)$. The equation is
\be{EL}
-\,v''-\frac{v'}r+\(\tfrac{B^2}4\,r^2+\alpha\)v=\mu_{\mathrm{EL}}(\alpha)\(\istwo{|v|^p}\)^{\frac2p-1}\,|v|^{p-2}\,v\,.
\ee
Without loss of generality we can restrict the problem to positive solutions such that
\[
\mu_{\mathrm{EL}}(\alpha)=\(\istwo{|v|^p}\)^{1-\frac2p}
\]
and then we have to solve the reduced problem
\[
-\,v''-\,\frac{v'}r+\(\tfrac{B^2}4\,r^2+\alpha\)v=|v|^{p-2}\,v
\]
among positive functions in $\mathrm H^1((0,+\infty),\,r\,dr)$ such that $\istwo{|v|^2}<+\infty$. From the existence result~\cite[Theorem~3.5]{MR1034014}, we know that $\mu_{\mathrm{EL}}(\alpha)$ is given by the infimum of $\big(\istwo{|v|^p}\big)^{1-2/p}$ on the set of solutions. Uniqueness and nondegeneracy of positive solutions to the above equation has been proved in~\cite{MR2332079} and~\cite{MR3470747}. Numerically, we solve the ODE on a finite interval, which induces a numerical error: the interval has to be chosen large enough, so that the computed value is a good upper approximation of $\mu_{\mathrm{EL}}(\alpha)$.\\
\emph{Case} (ii). Assume that $p\in(1,2)$. A radial minimizer of $\mathcal Q_\alpha^{(p)}$ solves
\[
-\,v''-\frac{v'}r+\tfrac{B^2}4\,r^2\,v=\nu_{\mathrm{EL}}(\beta)\,v-\beta\,\(\istwo{|v|^p}\)^{\frac2p-1}\,|v|^{p-2}\,v\,.
\]
Without loss of generality we can restrict the problem to positive solutions such that
\[
\beta=\(\istwo{|v|^p}\)^{1-\frac2p}
\]
and have therefore to solve the reduced problem
\be{EL2}
-\,v''-\,\frac{v'}r+\tfrac{B^2}4\,r^2\,v=\nu\,v-|v|^{p-2}\,v
\ee
among nonnegative functions in $\mathrm H^1((0,+\infty),\,r\,dr)$ such that $\istwo{|v|^p}<+\infty$. Notice that the \emph{compact support principle} applies according to, \emph{e.g.},~\cite{MR1938658,MR1387457,MR1629650,MR1715341}, since $p-1<1$ so that the nonlinearity in the right hand side of~\eqref{EL2} is non-Lipschitz. Numerically, we can therefore solve~\eqref{EL2} using a shooting method, with a shooting parameter $a=v(0)>0$ that has to be adjusted to provide a nonnegative solution with compact support, which minimizes $\istwo{|v|^p}$. The set of solutions is then parametrized by the parameter $\nu>0$, while $\beta$ is recovered by the above integral condition. In other words, we approximate $\nu\mapsto\beta_{\mathbf B}(\nu)$ and recover $\beta\mapsto\nu_{\mathbf B}(\beta)$ as the inverse of $\beta_{\mathbf B}$. Since we compute the size of the support of the approximated solution, there is no numerical error due to finite size truncation.

%%%%%%%%%%%%%%%%%%%%%%%%%%%%%%%%%%%%%%%%%%%%%%%%%%%%%%%%%%%%%%%%%%%%%%
\subsection{Numerical results}\label{Sec:NumRes}
We illustrate the \emph{Case} (i), $p\in(2,+\infty)$, by computing for $p=3$ and $B=1$, in dimension $d=2$, an approximation of $\alpha\mapsto\mu_{\mathbf B}(\alpha)$. Upper estimates $\mu_{\mathrm{Gauss}}(\alpha)\ge\mu_{\mathrm{EL}}(\alpha)\ge\mu_{\mathbf B}(\alpha)$ and lower estimates $\mu_{\mathrm{interp}}(\alpha)\le\mu_{\kern1pt\mathrm{LT}}(\alpha)\le\mu_{\mathbf B}(\alpha)$ are surprisingly close: see Figs.~\ref{Fig:F1} and~\ref{Fig:F2}.
%---------------------------------------------------------------------
\begin{figure}[ht]
\begin{center}
\includegraphics[width=9cm]{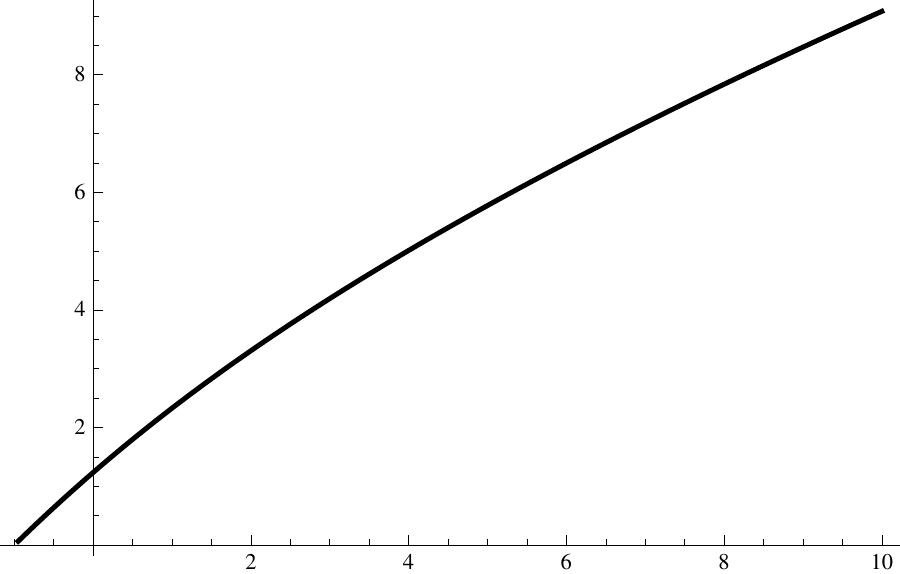}
\caption{\small\label{Fig:F1} Case $d=2$, $p=3$, $B=1$: plot of $\alpha\mapsto(2\,\pi)^{\frac2p-1}\,\mu_{\mathbf B}(\alpha)$.}
\end{center}
\end{figure}
%---------------------------------------------------------------------
\begin{figure}[ht]
\begin{center}
\includegraphics[width=9cm]{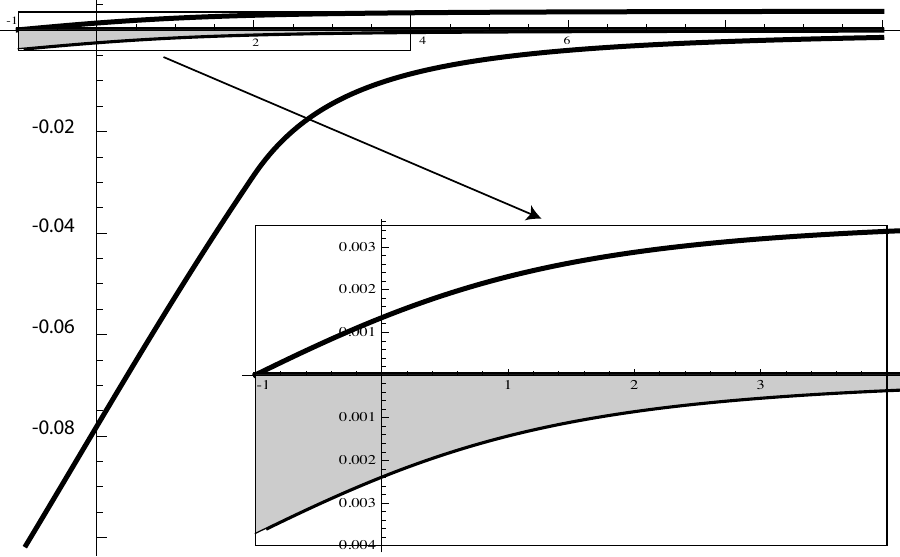}
\caption{\small\label{Fig:F2} Case $d=2$, $p=3$, $B=1$: comparison of the upper estimates $\alpha\mapsto\mu_{\mathrm{Gauss}}(\alpha)$ and $\alpha\mapsto\mu_{\mathrm{EL}}(\alpha)$ of Sections~\ref{Sec:upper} and~\ref{Sec:EL}, with the lower estimates $\alpha\mapsto\mu_{\mathrm{interp}}(\alpha)$ and $\alpha\mapsto\mu_{\kern1pt\mathrm{LT}}(\alpha)$ of Propositions~\ref{T1-gen} and~\ref{T1-gen-bis}. Plots represent the curves $\log_{10}(\mu_{\mathrm{Gauss}}/\mu_{\mathrm{EL}})$, $\log_{10}(\mu_{\kern1pt\mathrm{LT}}/\mu_{\mathrm{EL}})$ and $\log_{10}(\mu_{\mathrm{interp}}/\mu_{\mathrm{EL}})$ so that $\alpha\mapsto\mu_{\mathrm{EL}}(\alpha)$ corresponds to a straight line at level $0$. The exact value associated with $\mu_{\mathbf B}$ lies in the grey area.}
\end{center}
\end{figure}
%---------------------------------------------------------------------

In \emph{Case} (ii), $p\in(1,2)$, the range of the curve $\beta\mapsto\nu_{\mathbf B}(\beta)$ differs from the case $p>2$ but again upper estimates $\nu_{\mathrm{Gauss}}(\beta)\ge\nu_{\mathrm{EL}}(\beta)\ge\nu_{\mathbf B}(\beta)$ and lower estimates $\nu_{\mathrm{interp}}(\beta)\le\nu_{\kern1pt\mathrm{LT}}(\beta)\le\nu_{\mathbf B}(\beta)$ are surprisingly close: see Figs.~\ref{Fig:F3} and~\ref{Fig:F4}.
%---------------------------------------------------------------------
\begin{figure}[ht]
\begin{center}
\includegraphics[width=10cm]{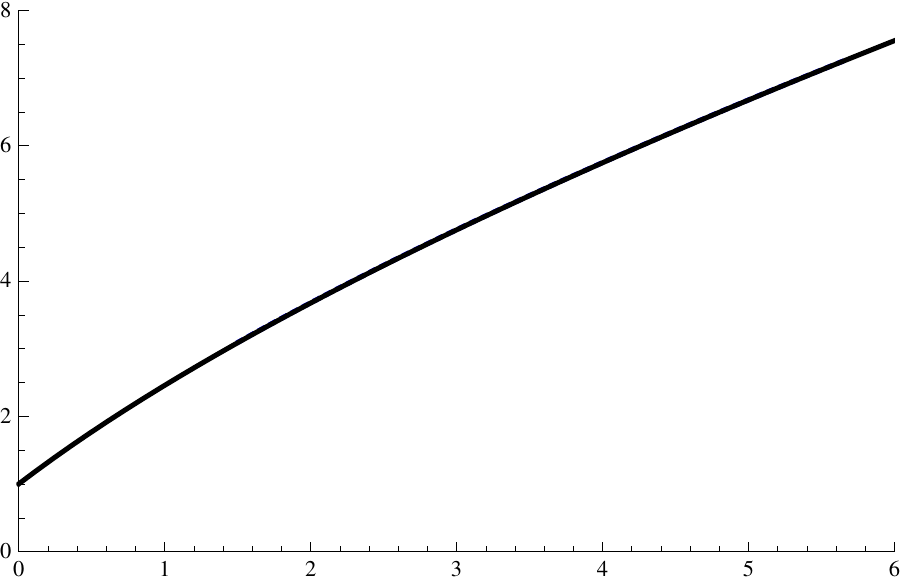}
\caption{\small\label{Fig:F3} Case $d=2$, $p=1.4$, $B=1$: plot of $\beta\mapsto\nu_{\mathbf B}(\beta)$. The horizontal axis is measured in units of $(2\,\pi)^{1-\frac2p}\,\beta$.}
\end{center}
\end{figure}
%---------------------------------------------------------------------
\begin{figure}[ht]
\begin{center}
\includegraphics[width=9cm]{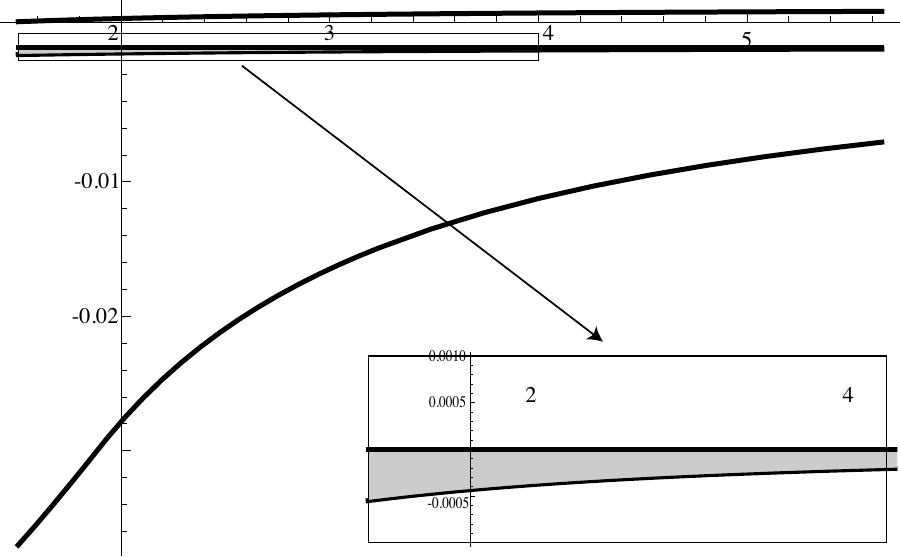}
\caption{\small\label{Fig:F4} Case $d=2$, $p=1.4$, $B=1$, with same horizontal scale as in Fig.~\ref{Fig:F3}: comparison of the upper estimates $\beta\mapsto\nu_{\mathrm{Gauss}}(\beta)$ and $\beta\mapsto\nu_{\mathrm{EL}}(\beta)$ of Sections~\ref{Sec:upper} and~\ref{Sec:EL}, with the lower estimates $\nu_{\mathrm{interp}}(\beta)$ and $\beta\mapsto\nu_{\kern1pt\mathrm{LT}}(\beta)$ of Propositions~\ref{T1-gen-qsmall} and~\ref{T1-gen-pless23}. Plots represent the curves $\log_{10}(\nu_{\mathrm{Gauss}}/\nu_{\mathrm{EL}})$, $\log_{10}(\nu_{\kern1pt\mathrm{LT}}/\nu_{\mathrm{EL}})$ and $\log_{10}(\nu_{\mathrm{interp}}/\nu_{\mathrm{EL}})$ so that $\alpha\mapsto\nu_{\mathrm{EL}}(\beta)$ corresponds to a straight line at level $0$. The exact value associated with $\nu_{\mathbf B}$ lies in the grey area.}
\end{center}
\end{figure}
%---------------------------------------------------------------------

%%%%%%%%%%%%%%%%%%%%%%%%%%%%%%%%%%%%%%%%%%%%%%%%%%%%%%%%%%%%%%%%%%%%%%
\subsection{Asymptotic regimes}\label{Sec:Asymptotic}
We investigate some asymptotic regimes in the case of a constant magnetic field of intensity $B$.

%%%%%%%%%%%%%%%%%%%%%%%%%%%%%%%
\subsubsection*{Convergence towards the Lowest Landau Level} Assume that $d=2$, $p>2$ and let us consider the regime as $\alpha\to(-B)_+$. We denote by \LLL\ the eigenspace corresponding to the \emph{Lowest Landau Level}.
%---------------------------------------------------------------------
\begin{prop} Let $d=2$ and consider a constant magnetic field with field stren\-gth~$B$. If $\psi_\alpha$ is a minimizer for $\mu_{\mathbf B}(\alpha)$ such that $\nrm{\psi_\alpha}p=1$, then there exists a non trivial $\varphi_\alpha\in$ \LLL\ such that
\[
\lim_{\alpha\to(-B)_+}\left\|\psi_\alpha-\varphi_\alpha\right\|_{\mathrm H^1_{\!\mathbf A}(\R^2)}=0\,.
\]
\end{prop}
%---------------------------------------------------------------------
\begin{proof} Let $\psi_\alpha\in\mathrm H^1_{\!\mathbf A}(\R^2)$ be an optimal function for ~\eqref{Interp1} such that $\nrm{\psi_\alpha}p=1$ and let us decompose it as $\psi_\alpha=\varphi_\alpha+\chi_\alpha$, where $\varphi_\alpha\in $ \LLL\ and $\chi_\alpha$ is in the orthogonal of \LLL. Then, by the orthogonality of $\varphi_\alpha$ and $\chi_\alpha$, we get
\[
\mu_{\mathbf B}(\alpha)\ge (\alpha+B)\,\nrm{\varphi_\alpha}2^2+(\alpha+3B)\,\nrm{\chi_\alpha}2^2\ge(\alpha+3B)\,\nrm{\chi_\alpha}2^2\sim2B\,\nrm{\chi_\alpha}2^2
\]
as $\alpha\to(-B)_+$ because $\nrm{\nabla\chi_\alpha}2^2\ge3B\,\nrm{\chi_\alpha}2^2$. Since $\lim_{\alpha\to(-B)_+}\mu_{\mathbf B}(\alpha)=0$ by Theorem~\ref{Thm:Interp}, this implies that $\lim_{\alpha\to(-B)_+}\nrm{\chi_\alpha}2=0$. On the other hand, we know that
\[
\mu_{\mathbf B}(\alpha)=(\alpha+B)\,\nrm{\varphi_\alpha}2^2+\,\nrm{\nabla_{\!\mathbf A}\,\chi_\alpha}2^2+\alpha\,\nrm{\chi_\alpha}2^2\ge\tfrac23\,\nrm{\nabla_{\!\mathbf A}\,\chi_\alpha}2^2\,,
\]
which concludes the proof.\end{proof}

%%%%%%%%%%%%%%%%%%%%%%%%%%%%%%%
\subsubsection*{Semi-classical regime} Let us consider the small magnetic field regime. We assume that the magnetic potential is given by~\eqref{B-d=2} if $d=2$. In dimension $d=3$, we choose $\mathbf A=\frac B2(-x_2,x_1,0)$ and observe that the constant magnetic field is $\mathbf B=(0,0,B)$, while the spectral gap in~\eqref{Gap} is $\Lambda[\mathbf B]=B$.
%---------------------------------------------------------------------
\begin{prop}\label{Prop:OptCst} Let $d=2$ or $3$ and consider a constant magnetic field $\mathbf B$ of intensity $B$ with magnetic potential $\mathbf A$ and assume that~\eqref{Gap} holds for some \hbox{$\Lambda=\Lambda[\mathbf B]>0$}.
\begin{enumerate}
\item[(i)] For $p\in(2,2^*)$ and for any fixed $\alpha$ and $\mu>0$, we have
\[
\lim_{\varepsilon\to0_+}\mu_{\varepsilon\,\mathbf B}(\alpha)=C_p\,\alpha^{\frac dp-\frac{d-2}2}\quad\mbox{and}\quad\lim_{\varepsilon\to0_+}\alpha_{\varepsilon\,\mathbf B}(\mu)=\(\mathsf C_p^{-1}\,\mu\)^\frac{2\,p}{2\,p-d\,(p-2)}\,.
\]
\item[(ii)] For $p\in(1,2)$ and any fixed $\beta>0$, we have
\[
\lim_{\varepsilon\to0_+}\nu_{\varepsilon\,\mathbf B}(\beta)=C_p\,\beta^\frac{2\,p}{2\,p+d\,(2-p)}\,.
\]
\end{enumerate}\end{prop}
%---------------------------------------------------------------------
\begin{proof} Consider any function $\psi\in\mathrm H^1_{\!\mathbf A}(\R^d)$ and for any $\varepsilon>0$ define $\psi(x)=\chi(\sqrt\varepsilon\,x)$. With our standard choice for $\mathbf A$, we have that $\sqrt\varepsilon\,\mathbf A\big(x/\sqrt\varepsilon\big)=\mathbf A(x)$. From
\[
\frac{\nrm{\nabla_{\!\varepsilon\,\mathbf A}\psi}2^2+\alpha\,\nrm\psi2^2}{\nrm\psi p^2}=\varepsilon^{\frac dp-\frac{d-2}2}\,\frac{\nrm{\nabla_{\!\mathbf A}\chi}2^2+\alpha\,\varepsilon^{-1}\,\nrm\chi2^2}{\nrm\chi p^2}\,,
\]
we deduce that
\[
\mu_{\varepsilon\,\mathbf B}(\alpha)=\varepsilon^{\frac dp-\frac{d-2}2}\,\mu_{\mathbf B}\(\alpha\,\varepsilon^{-1}\)\,.
\]
By a similar argument we can easily see that
\[
\alpha_{\varepsilon\,\mathbf B}(\mu)=\varepsilon\,\alpha_{\mathbf B}\(\mu\,\varepsilon^{-\frac{2\,q-d}{2\,q}}\)\quad\mbox{and}\quad\nu_{\varepsilon\,\mathbf B}(\beta)=\varepsilon\,\nu_{\mathbf B}\(\beta\,\varepsilon^{\frac{d-2}2-\frac dp}\)\,.
\]
The conclusion follows by considering the asymptotic regime as $\varepsilon\to0_+$ in Theorem~\ref{Thm:Interp} and in Corollary~\ref{Cor:KLT1}.\end{proof}

%%%%%%%%%%%%%%%%%%%%%%%%%%%%%%%%%%%%%%%%%%%%%%%%%%%%%%%%%%%%%%%%%%%%%%
\subsection{A numerical result on the linear stability of radial optimal functions}
Bonheure et al. show in~\cite{2016arXiv160700170B} that for a fixed $\alpha>0$ and for $\mathbf B$ small enough, the optimal functions for~\eqref{Interp1} are radially symmetric functions, \emph{i.e.}, belong to $\mathcal C_0$. As shown in Proposition~\ref{Prop:OptCst}, this regime is equivalent to the regime as $\alpha\to+\infty$ for a given $\mathbf B$, at least if the magnetic field is constant. On the other hand, the numerical results of Section~\ref{Sec:Numerics} show that $\alpha\mapsto\mu_{\mathbf B}(\alpha)$ is remarkably well approximated from above by functions in $\mathcal C_0$. The approximation from below of Proposition~\ref{T1-gen-pless23}, although not exact, is found to be numerically very close.

This raises the open question of whether, in the case of constant magnetic fields, equality in~\eqref{Interp1} is realized by radial functions for a given constant magnetic field $\mathbf B$ and an arbitrary $\alpha$. As  mentioned in Section~\ref{Sec:EL}, from~\cite{MR2332079,MR3470747}, we know that the branch of solutions in $\mathcal C_0$ is isolated in the class of radial functions. Perturbing these radial solutions in a larger class of functions is natural. Let us analyze the stability of the solutions to~\eqref{EL} under perturbations by functions in $\mathcal C_1$. Assume that $d=2$ and $p>2$. Let us denote by $\psi_0$ a minimizer of $\mathcal Q_\alpha^{(p)}$ on the class $(\mathcal C_0)$ of radial functions, normalized so that, with a standard abuse of notation, $\psi_0(x)=\psi_0(|x|)$ solves
\[
-\,\psi_0''-\,\frac{\psi_0'}r+\(\tfrac{B^2}4\,r^2+\alpha\)\psi_0=|\psi_0|^{p-2}\,\psi_0\,,
\]
and consider the test function
\[
\psi_\varepsilon=\psi_0+\varepsilon\,e^{i\,\theta}\,v
\]
where $v$ is a radial function, depending only on $r=|x|$, and $e^{i\,\theta}=(x_1+i\,x_2)/r$. In the asymptotic regime as $\varepsilon\to0_+$, we have
\begin{multline*}
\irtwo{|\nabla_{\!\mathbf A}\psi_\varepsilon|^2}+\alpha\,\irtwo{|\psi_\varepsilon|^2}-\(\irtwo{|\nabla_{\!\mathbf A}\psi_0|^2}+\alpha\,\irtwo{|\psi_0|^2}\)\\
=\(\irtwo{|\nabla_{\!\mathbf A}v|^2}+\alpha\,\irtwo{|v|^2}\)\varepsilon^2+o(\varepsilon^2)\\
=2\,\pi\istwo{\left[|v'|^2+\(\(\tfrac1r-\tfrac{B\,r}2\)^2+\alpha\)|v|^2\right]}\,\varepsilon^2+o(\varepsilon^2)
\end{multline*}
and
\[
\nrm{\psi_\varepsilon}p^2-\nrm{\psi_0}p^2=2\,\pi\,\tfrac p2\,\nrm{\psi_0}p^{2-p}\(\istwo{|\psi_0|^{p-2}\,v^2}\)\varepsilon^2+o(\varepsilon^2)\,.
\]
Altogether, we obtain
\[\begin{array}{l}
\(\mathcal Q_\alpha^{(p)}[\psi_\varepsilon]-\mu_0(\alpha)\)\nrm{\psi_0}p^2\\
=2\,\pi\left[\irtwo{|v'|^2}+\irtwo{\(\(\tfrac1r\!-\!\tfrac{B\,r}2\)^2\!+\alpha\)|v|^2}-\tfrac p2\istwo{|\psi_0|^{p-2}\,v^2}\right]\varepsilon^2+o(\varepsilon^2)
\end{array}\]
where $\mu_0(\alpha)=\nrm{\psi_0}p^{p-2}=\mathcal Q_\alpha^{(p)}[\psi_0]$. The linear stability of $\psi_0$ with respect to perturbations in $(\mathcal C_1)$ can be recast as the eigenvalue problem
\be{Perturbation}
-\,v''-\,\frac{v'}r+\(\(\tfrac1r-\tfrac{B\,r}2\)^2+\alpha\)v-\tfrac p2\,|\psi_0|^{p-2}\,v=\mu\,v\,.
\ee
The numerical results for $d=2$, $B=1$ and $p=3$ of Fig.~\ref{Fig:F5} suggest that $\mathcal Q_\alpha^{(p)}$ is linearly stable for $\alpha>-B$, not too large. This indicates that $\mu_{\mathrm{EL}}$ is a good candidate for computing the exact value of $\mu_{\mathbf B}$ for arbitrary values of $B$'s.
%---------------------------------------------------------------------
\begin{figure}[ht]
\includegraphics[width=7cm]{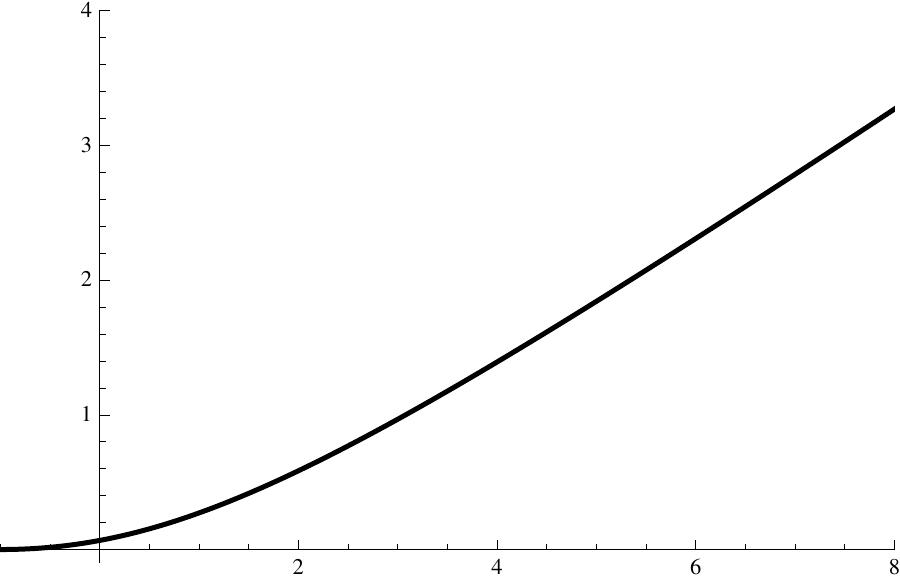}
\caption{\small\label{Fig:F5} Case $p=3$ and $B=1$: plot of $\mu$ solving~\eqref{Perturbation} as a function of $\alpha$. A careful investigation shows that $\mu$ is always positive, including in the limiting case as $\alpha\to(-B)_+$, thus proving the numerical stability of the optimal function in $\mathcal C_0$ with respect to perturbations in $\mathcal C_1$.}
\end{figure}
%---------------------------------------------------------------------

%%%%%%%%%%%%%%%%%%%%%%%%%%%%%%%%%%%%%%%%%%%%%%%%%%%%%%%%%%%%%%%%%%%%%%
%%%%%%%%%%%%%%%%%%%%%%%%%%%%%%%%%%%%%%%%%%%%%%%%%%%%%%%%%%%%%%%%%%%%%%
%\clearpage
%\nocite*
\vspace*{-0.5cm}
%\bibliographystyle{siam}
%\bibliography{DoEsLaLo}

%%%%%%%%%%%%%%%%%%%%%%%%%%%%%%%%%%%%%%%%%%%%%%%%%%%%%%%%%%%%%%%%%%%%%%
%%%%%%%%%%%%%%%%%%%%%%%%%%%%%%%%%%%%%%%%%%%%%%%%%%%%%%%%%%%%%%%%%%%%%%
\end{document}